\documentclass[reqno]{amsart}
\usepackage{amssymb}
\usepackage[usenames, dvipsnames]{color}
\usepackage{txfonts}
\usepackage{hyperref}
\usepackage{verbatim}
\usepackage[pagewise]{lineno}
\usepackage{marginnote}
\usepackage{todonotes}
\usepackage[normalem]{ulem}
\usepackage{cancel}
\usepackage{soul}

\theoremstyle{plain}
\newtheorem{theorem}{Theorem}[section]
\newtheorem{lemma}[theorem]{Lemma}
\newtheorem{corollary}[theorem]{Corollary}
\newtheorem{proposition}[theorem]{Proposition}

\theoremstyle{definition}
\newtheorem{definition}[theorem]{Definition}

\theoremstyle{remark}

\newcommand{\bR}{{\mathbb R}}
\newcommand{\bN}{{\mathbb N}}

\newcommand{\pf}{\noindent {\bf Proof. \hspace{2mm}}}
\def\ve{\varepsilon}
\def\lam{\lambda}

\def\t{\tilde}
\def\th{\theta}

\def\G{\Gamma}
\def\dl{\delta}
\def\Dl{\Delta}
\def\lt{\left}
\def\rt{\right}

\def\i{\infty}
\def\e{\epsilon}

\def\sgn{\text{sgn }}
\def \ls{\lesssim}
\def\p{\partial}
\def\f{\frac}
\def\na{\nabla}
\def\al{\alpha}

\def\O{\Omega}
\def\o{\omega}

\def\q{\quad}
\def\qq{\qquad}

\def\blue{\color{blue}}
\def\be{\begin{equation}}
\def\ee{\end{equation}}
\def\bes{\begin{equation*}}
\def\ees{\end{equation*}}
\def\bali{\begin{aligned}}
\def\eali{\end{aligned}}
\def\lab{\label}
\def\bS{\bf S}

\def\2O{\underline{\O}}
\def\mO{\mathcal{O}}
\def\mC{\mathcal{C}}

\numberwithin{equation}{section}

\makeatletter
\def\dashint{\operatorname%
{\,\,\text{\bf--}\kern-.98em\DOTSI\intop\ilimits@\!\!}}
\makeatother

\begin{document}

%\linenumbers

\title[Review on axially symmetric NS]{A Review of results on  axially symmetric Navier-Stokes equations, with addendum by X. Pan and Q. S. Zhang}

%\author[X. Pan]{Xinghong Pan}
%\address[X. Pan]{Department of Mathematics, Nanjing University of Aeronautics and Astronautics, Nanjing 211106, China}

\thanks{Xinghong Pan is supported by Natural Science Foundation of Jiangsu Province (No. BK20180414), Double Innovation Scheme of Jiangsu Province and National Natural Science Foundation of China (No. 11801268).}

\author[Q. S. Zhang]{Qi S. Zhang}

\address[Xinghong Pan]{Department of Mathematics, Nanjing University of Aeronautics and Astronautics, Nanjing 211106, China.}

\email{xinghong\_87@nuaa.edu.cn}

\address[Q. S. Zhang]{Department of mathematics, University of California, Riverside, CA 92521, USA}

\email{qizhang@math.ucr.edu}



\begin{abstract}
In this paper, we give a brief survey of recent results on axially symmetric Navier-Stokes equations (ASNS) in the following categories: regularity criterion, Liouville property for ancient solutions, decay and vanishing of stationary solutions. Some discussions also touch on the full 3 dimensional equations.
Two results, closing of the scaling gap for ASNS and vanishing of homogeneous D solutions in 3
 dimensional slabs will be described in more detail.

 In the addendum, two new results in the 3rd category will also be presented, which are generalizations of recently published results by the author and coauthors.

\end{abstract}
%\today
\maketitle

\tableofcontents

\section{Introduction}

The Cauchy problem of  Navier-Stokes equations (NS) describing the motion of viscous incompressible fluids in $\bR^3$ is
\be
\lab{nse}
\begin{cases}
&\mu \Delta v -  {\blue ( v\cdot \nabla)} v - \nabla P -\partial_t v =0, \quad \text{on} \quad \bR^3 \times (0, \infty)\\
&div \, v=0,  v(x, 0)= v_0(x).
\end{cases}
\ee Here $v$ is the velocity field,  $P$ is the pressure, both of which are the unknowns;
$v_0$ is the given initial velocity; $\mu>0$ is the viscosity constant, which will be taken as $1$ unless stated otherwise. One can also add a forcing term on the righthand side, then it becomes a nonhomogeneous problem.

Thanks to Leray's work \cite{Le2} in 1934, we know the above problem has a weak solution $v \in L^\infty((0, \infty), L^2(\bR^3))$ such that $|\nabla v| \in L^2((0, \infty), L^2(\bR^3))$  provided that the initial condition has finite kinetic energy. Moreover, $\Vert v(\cdot, t) -  v_0(\cdot) \Vert_{L^2(\bR^3)} \to 0$ as
$t \to 0$ and  $\forall \,  T>0$,
\be
\lab{enineq}
\int |v(x, T)|^2 dx + 2 \int^T_0\int |\nabla v(x, t)|^2 dxdt \le \int |v_0(x)|^2 dx < \infty.
\ee See Theorem 3.10 in Tsai's book \cite{Tsai} for a modern and concise proof e.g.
Solutions satisfying \eqref{enineq} are often referred to as Leray-Hopf solutions, in order to distinguish them from even weaker solutions.
In general, one does not know if a Leray-Hopf solution is smooth, except for a few special cases, usually as a perturbation of a special smooth solution. Stability of NS under small perturbation is well studied. A general result of such kind can be found in \cite{RRST} e.g.
Over the years several sufficient conditions under which Leray-Hopf solutions are smooth   have been obtained. For example the Ladyzhenskaya-Prodi-Serrin condition: $|v| \in L^{p, q}_{x, t}$ with $\frac{3}{p}+\frac{2}{q} \le 1$ and $3<p<\infty$ and the end point result $p=3, q=\infty$ by Escauriaza, Seregin and Sverak \cite{ESS}. See also \cite{DoWa} by Dong and Wang in higher dimensional cases, including both interior and boundary regularity. Here and later, a measurable function $f=f(x, t)$ is said to be in $L^{p, q}_{x, t}$ if $\Vert f \Vert_{L^{p, q}_{x, t}} \equiv \left(\int^\infty_0\left(\int_{\bR^3} |f|^p dx\right)^{q/p} dt \right)^{1/q} <\infty$.
If $\frac{3}{p}+\frac{2}{q} = 1$, these conditions are scaling invariant or critical under the natural scaling of the Navier Stokes equations: for $\lam>0$, if $(v, P)$ solves \eqref{nse}, then $(v_\lam, P_\lam)$ defined by
\be
\lab{nsscal}
v_{\lambda}(x, t) \equiv \lambda v(\lam x, \lam^2 t), \qquad P_\lam(x, t) \equiv \lam^2 P(\lam x, \lam^2 t)
\ee also solves \eqref{nse}. It is easy to see that $\Vert v \Vert_{L^{p, q}_{x, t}} =
\Vert v_\lam \Vert_{L^{p, q}_{x, t}}$ for the above $p, q$.
 Sometimes these conditions can be improved logarithmically, even for endpoint cases. See the articles X.H. Pan \cite{Panx}, T. Tao \cite{Tao}, Barker and Prange \cite{BaPr}, e.g.
A partial regularity result for the so-called "suitable weak solutions" was found by Caffarelli, Kohn and Nirenberg \cite{CKN:1}, building on earlier work of Scheffer \cite{Schef1, Schef2}. These solutions are Leray-Hopf solutions with an extra integrability condition on the pressure term $P$. It is proven that the singular set of suitable weak solutions, if exists, has one dimensional parabolic Hausdorff measure $0$.
The proof utilizes a blow up argument to deduce an  $\epsilon$ regularity result: smallness of certain scaling invariant integral quantities involving the velocity or vorticity  implies boundedness of solutions. Then the size estimate of the possible singular set follows from a covering argument.  See also the papers of F. H. Lin \cite{Linf}, A. Vasseur \cite{Vass}, J. Wolf \cite{Wfj} for similar results and shorter proofs, some of which employ a De Giogi type (refined energy) method instead of blow up method.
In \cite{Wfj}, using the decomposition ${\blue ( v\cdot \nabla)v}  = \frac{1}{2} \nabla |v|^2 - (\nabla \times v) \times v$, it is observed that one can also use a scaling invariant quantity involving $(\nabla \times v) \times \frac{v}{|v|}$. One consequence is that small perturbations of Beltrami flows are regular. Recall that a flow or a vector field $v$ is called a Beltrami flow if the vorticity $\nabla \times v$ is parallel to $v$.
These conditional regularity results are consistent with the standard linear theory for second order parabolic equations with lower order coefficients, coming from the De Giorgi-Nash-Moser theory. As far as the regularity conditions are concerned, the nonlinearity of the NSE only induces a marginal improvement over the linear case.  One can see the difficulty in proving regularity by observing that the energy inequality \eqref{enineq} only tells us, after using Sobolev inequality and interpolation, that
$v \in L^{10/3, 10/3}_{x, t}$. In a local space time domain, this a priori bound is much weaker than the
regularity conditions mentioned above, such as $L^{5, 5}_{x, t}$.
If we hope to prove regularity of the solution, we need to study the behavior of $v$ at micro or infinitesimal scale. This is amount to studying the behavior of $v_\lam$ at a fixed scale while letting $\lam \to 0$. Notice that
\[
\Vert v_\lam \Vert_{L^{10/3, 10/3}_{x, t}} = \lam^{-1/2} \Vert v \Vert_{L^{10/3, 10/3}_{x, t}} \to \infty, \qquad \lam \to 0.
\]So the energy inequality does not furnish any information on micro scale. For this reason NSE is considered as a super critical equation, i.e. a priori estimate is weaker than regularity condition at small scale. Another way to see the super criticality is to consider the vorticity $\o=\nabla \times v$ which satisfies the equation
\be
\lab{eqvor}
\Delta \o - {\blue ( v\cdot \nabla)} \o + {\blue ( \o\cdot \nabla)} v - \p_t \o=0.
\ee The energy inequality \eqref{enineq} tells us $| \nabla v | \in L^{2, 2}_{x, t}$. This again is much weaker than the regularity condition of $L^{p, q}_{x, t}$ with $ \frac{3}{p}+\frac{2}{q} < 2$,  if $\nabla v$ are regarded as  potential functions for the vorticity equation. We note that the drift term $v \nabla \o$, can be handled even locally by an integral argument since $ div \, v=0$, c.f. \cite{Zcmp04}.
So the most dangerous and mysterious term in the vorticity equation is the so-called vortex stretching term $\o \nabla v$. An avenue of attack on the regularity problem is to exploit the structure on the sets where the vorticity has high value, such as angles, intermittency  and sparseness. See \cite{CoFe}, \cite{DVB}, \cite{GrZh} and a recent \cite{BFG} e.g. In \cite{BFG}  Bradshaw, Farhat and Grujic  reduce the scaling gap by using sparseness of the super level sets of the positive and negative parts of the vorticity components at a scale comparable to the sup of $1/|\o|$ nearby. There are also many activities on
one component regularity conditions and regularity conditions on directional derivatives of the velocity:
 \cite{NeNoPe}, \cite{Zhoyo}, \cite{ChemZh},  \cite{ChZZ}, \cite{HLLZ}; \cite{PePo}, \cite{KuZi}, \cite{CaTi} e.g. One-sided conditions can also be imposed on the eigenvalues of $\nabla v$ or the middle eigenvalue of the
 strain tensor $[\nabla v + (\nabla v)^T]/2$. See \cite{Zhang2006} and \cite{Mill} e.g.
However, unless something dramatic happens, such as the discovery of a critical a priori estimate, a magic cancellation, or an
ingenious construction of a blow up solution, the regularity problem for the 3 dimensional NSE
 \eqref{nse} will remain open. Even if some Leray-Hopf solutions are found to blow up in finite time, it is still interesting to characterize the set
of initial values that give rise to Leray-Hopf solutions that stay smooth all time, which is nonempty.

In this paper, we will focus on a special case of \eqref{nse}, namely when $v$ and $P$ are independent of the angle in a cylindrical coordinate system $(r,\,\th,\,x_3)$. That is, for $x=(x_1,\,x_2,\,x_3) \in \bR^3$,
$
r=\sqrt{x_1^2+x_2^2}, \quad
\th=\arctan (x_2/x_1),
$ and the basis vectors $e_r,e_\th,e_3$:
\[
e_r=(x_1/r, x_2/r, 0),\quad e_\th=(-x_2/r, x_1/r, 0),\quad e_3=(0,0,1);
\]
 and  solutions are given by
\[
v=v^r(r,x_3, t)e_r+v^{\th}(r, x_3, t)e_{\th}+v^3(r, x_3, t)e_3.
\]

By direct computation, $v^r$, $v^3$ and $v^\theta$ satisfy the axially symmetric Navier-Stokes equations
\be
\begin{aligned}
\lab{eqasns}
\begin{cases}
   \big (\Delta-\frac{1}{r^2} \big )
v^r-(v^r \p_r + v^3 \p_{x_3})v^r+\frac{(v^{\theta})^2}{r}-\partial_r
P-\p_t  v^r=0,\\
   \big   (\Delta-\frac{1}{r^2}  \big
)v^{\theta}-(v^r \p_r + v^3 \p_{x_3} )v^{\theta}-\frac{v^{\theta} v^r}{r}-
\partial_t v^{\theta}=0,\\
 \Delta v^3-(v^r \p_r + v^3 \p_{x_3})v^3-\p_{x_3} P-\p_t v^3=0,\\
 \frac{1}{r} \p_r (rv^r) +\p_{x_3}
v^3=0,
\end{cases}
\end{aligned}
\ee which will be abbreviated as ASNS. It looks more complicated than the full 3 dimensional equation.
 But simplifications happen in the 2nd equation where the pressure term disappeared. A tip in carrying out vector calculations under the cylindrical system
is to use tensor notations. For example
\[
\bali
\nabla v &= \p_r v \otimes e_r +\frac{1}{r} \p_\th v \otimes e_\th + \p_{x_3} v \otimes e_3\\
&=(\p_r v^r  e_r + \p_r v^\th e_\th + \p_r v^3 e_3 )\otimes e_r + \frac{1}{r} ( v^r e_\th -
v^\th e_r) \otimes e_\th \\
&\qquad \qquad + (\p_{x_3} v^r  e_r + \p_{x_3} v^\th e_\th + \p_{x_3} v^3 e_3 )\otimes e_3.
\eali
\]Taking the inner product with the second entry,  the convection terms become
\[
\bali
&-{\blue ( v\cdot \nabla)} v = -(v^r e_r+v^{\th}e_{\th}+v^3 e_3) \cdot (\p_r v^r  e_r + \p_r v^\th e_\th + \p_r v^3 e_3 )\otimes e_r\\
&\qquad - (v^r e_r+v^{\th}e_{\th}+v^3 e_3) \cdot \frac{1}{r} ( v^r e_\th -
v^\th e_r) \otimes e_\th\\
&\qquad - (v^r e_r+v^{\th}e_{\th}+v^3 e_3) \cdot (\p_{x_3} v^r  e_r + \p_{x_3} v^\th e_\th + \p_{x_3} v^3 e_3 )\otimes e_3\\
&=[-(v^r \p_r + v^3 \p_{x_3})v^r+\frac{(v^{\theta})^2}{r}] e_r
-[(v^r \p_r + v^3 \p_{x_3} )v^{\theta}+\frac{v^{\theta} v^r}{r}] e_\th
-(v^r \p_r + v^3 \p_{x_3})v^3 e_3.
\eali
\]This gives rise to the most complicated terms in \eqref{eqasns}.

In 2015, it was observed  by Lei-Zhang \cite{LZ17} that ASNS is essentially critical under the standard scaling. So the aforementioned scaling gap is $0$. This observation has the effect of making ASNS looks less formidable than the full 3 dimensional case which has a positive scaling gap. Nevertheless all major open problems for the latter are still open for the former. In the next three sections, we will describe some recent research results on the following topics: regularity conditions (Section 2), ancient solutions (Section 3)  and stationary solutions (Section 4). These topics are closely related.
The study of possible singularity of solutions leads to the study of ancient solutions, i.e. solutions whose existence time extends to $-\infty$.
Stationary solutions are ancient solutions but are even more special. The aforementioned topics have an apparently decreasing order in terms of level of difficulties:

{\it regularity problem $>$ classification of ancient solutions $>$ classification of homogeneous stationary solutions $>$
classification of homogeneous D solutions.}

The last type of solutions are defined as stationary solutions to \eqref{nse} with finite Dirichlet energy:
$\int_{\bR^3} |\nabla v|^2 dx <\infty$, which also vanish at infinity. See Definition \ref{deDS}. They were studied in Leray's first paper
\cite{Le1}.
However as we shall see, even our understanding of stationary solutions are still very primitive. For example, we still do not know if homogeneous D solutions in $\bR^3$ are zero, even in the axially symmetric case.

Due to the large number of papers in the literature, we need to make a selection on what to present. This selection only reflects  personal interest and knowledge. Some important results may be missed.
For example, we will not address any papers on boundary value problems seriously, which as well known, have their own
complications and complexity. Nor will we touch non-uniquess of very weak solutions. Over the decades, in conjunction with the development of research on NS, many books have been written, and it is safe to say that the trend will continue. Let us list a few of them  for comprehensive information and the history of the field: \cite{Lady, Tem, CoFo, MaBe, KrLo, Can, Soh, Galdi2011, Leri1, Leri2, BoFa, Sereboo, Trie, RRS, Tsai}.

 Any suggestions on missing information or improvement are very welcome.

We end the introduction by listing a number of notations and conventions to be used throughout, which are more or less standard.
The velocity field is usually called $v$ and the vorticity $\nabla \times v$ is called $\o$. We use superscripts to denote their components in coordinates. Given a point $x=(x_1, x_2, x_3) \in \bR^3$,
we write $x'=(x_1, x_2, 0)$, $x_h=(x_1, x_2)$, $r =(x^2_1+x^2_2)^{1/2}$ and $\th=\arctan (x_2/x_1)$. $L^p(D)$, $p \ge 1$, denotes the usual Lebesgue space on a domain D which may be a spatial, temporal or space-time domain. Let $X$ be a Banach space defined for functions on $D \subset \bR^3$. $L^p(0, T; X)$ is the Banach space of space-time functions $f$ on the space time domain $D \times [0, T]$ with the norm $\left(\int^T_0 \Vert f(\cdot, t) \Vert^p_X dt\right)^{1/p}$. If no confusion arises, we will also use $L^p X$ to abbreviate $L^p(0, T; X)$. Sometimes we will also use $L^p_xL^q_t$ or $L^q_tL^p_x$ to denote the mixed $p, q$ norm in space time. Let $D \subset \bR^3$ be an open domain, then $H^1(D)=W^{1, 2}(D)=\{f \, | \,  f, |\nabla f| \in L^2(D) \}$ and $H^2(D)=W^{2, 2}(D)=\{f \, |\,  f, |\nabla f|,  |\nabla^2 f|\in L^2(D) \}$, the standard Sobolev spaces on D. Also, interchangeable notations $div \, v = \nabla \cdot v$, $v \nabla v =\sum v_i \p_{x^i} v = v \cdot \nabla v$ will be used. If there is no confusion, the vertical variable $x_3$ may be replaced with $z$. Also $B_r(x)$ denotes the ball of radius $r$ centered at $x$ in a Euclidean space. If $s$ is a number, then $s^-$ means any number which is close but strictly less than $s$.

\section{Regularity  Criterion }

\subsection{Critical and slightly super critical regularity conditions}

If the swirl $v^\theta=0$, then it is
known for long time ( O. A. Ladyzhenskaya
\cite{L}, M. R. Uchoviskii and B. I. Yudovich \cite{UY}), that finite energy solutions
to (\ref{eqasns}) are smooth for all time.
See also the paper by S.
Leonardi, J. Malek, J. Necas,  and  M. Pokorny \cite{LMNP}. By finite energy, we mean
\eqref{enineq} holds, i.e., the solution is a Leray-Hopf solution.

In the presence of swirl, it is still not known in general if finite energy solutions
blow up in finite time.

  However a lower bound for the possible blow up rate is known by
the results of
C.-C. Chen, R. M. Strain, T.-P.Tsai,  and  H.-T. Yau in \cite{CSTY1},
\cite{CSTY2},  G. Koch, N. Nadirashvili, G. Seregin,  and  V.
Sverak in \cite{KNSS}, which appeared around 2008. See also the work by G. Seregin  and
V. Sverak \cite{SS} for a localized version. These authors prove that if
\be
\lab{v<1/r}
|v(x, t)| \le
\frac{C}{r},
\ee then solutions are smooth for all time. Here $C$ is any positive constant.

The proof is based on the fact that the scaling invariant quantity
$\Gamma= r v^\theta$: satisfies the equation
\be
\lab{eqvth}
\Delta \Gamma - {\blue (b\cdot \nabla)} \Gamma- \frac{2}{r} \p_r
\Gamma-\p_t \Gamma=0,
\ee where $b=v^r e_r + v^3 e_3$. The bound \eqref{v<1/r} says that the equation is essentially scaling invariant and the classical linear regularity theory can be applied after some nontrivial modification.
In \cite{CSTY2}, the authors use this approach to prove that $\Gamma$ is H\"older continuous first. This implies $|v^\theta|$ is bounded by $r^{-1+\alpha}$ near the $z$ axis, which makes $v^\theta$ subcritical under the standard scaling. Here $\alpha$ is a small positive constant. Thus $v^\th$ is small in micro scale. The smallness enters into the equation for $\o^\theta$ in \eqref{eqvort} after a scaling argument. The authors then manage to prove that $\o^\th$ is bounded which in turn proves the whole velocity field $v$ is bounded by the Biot-Savart law.
In contrast, a blow up method is used in \cite{KNSS}. The first step is to show that if a solution of the ASNS blows up in finite time, then after a suitable scaling and limiting procedure, one obtains a nonzero bounded, mild solution of the ASNS, which exists in the time interval $(-\infty, 0]$. Such a solution, denoted by $v_\infty$,  is called a mild ancient solution.  Here the word "mild" means that a solution,
in addition to being a pointwise or weak solution of the Navier-Stokes equation, must also satisfy an integral equation involving the Stokes kernel. See \eqref{demildso} or \cite{KNSS} for a precise definition e.g. The purpose is to rule out solutions of the form $a(t) \nabla h(x)$ where $a=a(t)$ is a differentiable function and $h$ is a harmonic function.  The equation for $\Gamma=r v^\th_\infty$ again plays an essential role. Using the bound \eqref{v<1/r}, which survives the scaling,
and an integration argument involving the maximum principle for equation \eqref{eqvth}, it is shown that $v^\th_\infty=\Gamma/r=0$. Therefore $v_\infty$ is a swirl free, bounded, mild, ancient solution. Now the equation for $\o^\th/r$ (see \eqref{eqjoo}) satisfies the maximum principle.
Using the aforementioned integration argument for the equation of $\o^\th/r$, exploiting the interplay between the velocity and vorticity, one can show that $v_\infty$ is a constant. The bound \eqref{v<1/r} then forces $v_\infty=0$. But we already know that $v_\infty \neq 0$ from the  first step.  This contradiction shows that the original solution can not blow up.

 Solutions satisfying the bound \eqref{v<1/r} are
often referred to as type I solutions.
One reason for this name is that the bound is scaling invariant so one can study the solution at micro scales using available linear theory.
The above result can be summarized as: type I solutions of ASNS are regular.

Two years later, in the paper \cite{LZ11} by Lei and Zhang, it is proven that  if
$v^r, v^3$ are in the space of $L^\infty([0, \infty),  BMO^{-1}({\bR}^3))$ and $r v^\th(\cdot, 0) \in L^\infty$, then the solution is regular. Here $BMO$ is the space of functions with bounded mean oscillation, c.f. \cite{JoNi}, and $BMO^{-1}$ is the space of tempered distributions which can be written as partial derivatives of BMO functions. Well-posedness and other properties of solutions to NS have been studied by Koch-Tataru \cite{KoTa}, Miura \cite{Miu} and Germain-Pavlovic-Staffilani
\cite{GPS}.
Note the function $C/r$ is contained in $BMO^{-1}$. Hence this result extends the one described in the previous paragraph. See also \cite{Seregin2012}, \cite{WaZh12}. At the first glance, it seems that the main improvement is just the relaxation of a pointwise condition to an integral type condition. However, there is an additional feature in that one only needs to impose a condition on the vertical velocity $v^3$ to gain regularity. See Theorem \ref{thasnsvz} below. The precise statement is:

\begin{theorem} (\cite{LZ11})
\label{thregcon}
Let $v=v(x, t)$ be a Leray-Hopf solution to (\ref{eqasns}) in
the space time region ${\bR}^3 \times [0, T]$. Assume  that the
initial value satisfies, $|r
v^\theta(x, 0)|<C$. Suppose also $v(\cdot, t) = \nabla \times
B(\cdot, t)$ with $\sup_{0 < t < T}\|B(\cdot, t)\|_{{\rm BMO}}
\leq C_\ast$. Then $v$ is smooth in ${\bR}^3 \times (0, T]$.
Here $C$ and $C_\ast$ are arbitrary positive constants.
\end{theorem}

In the original paper, the pertinent solutions are stated as suitable weak solutions explained in Section 1.
However, no such restriction is actually needed in the proof. Another regularity condition proposed in \cite{LZ11} only involves a region outside a paraboloid with the vertex at a given space-time point.
This suggests that regularity of a solution at a space-time point only depends on the behavior of the solution in a small part of a space-time cube with vertex at the same point. Such a phenomenon was later
proven for the full 3 dimensional NS by Neustupa in \cite{Neu}.

Recently Seregin and Zhou \cite{SeZh} have relaxed the $L^\infty BMO^{-1}$ assumption further to $L^\infty \dot{B}^{-1}_{\infty, \infty}$ assumption.
Let us recall $\dot{B}^{-1}_{\infty, \infty}$ is the Besov space consisted of tempered distributions $f$ such that the norm
\[
\Vert f \Vert_{\dot{B}^{-1}_{\infty, \infty}} = \sup_{t>0} t^{1/2} \sup_{x} \left| \int_{\bR^3} G(x, t, y) f(y) dy \right|
\]is finite. Here $G(x, t, y)=(1/(4\pi t)^{3/2}) \exp(-|x-y|^2/(4t))$ is the standard heat kernel on $\bR^3$.

\begin{theorem} (\cite{SeZh})
 Any axially symmetric suitable weak solution of \eqref{eqasns}, belonging to $L^\infty \dot{B}^{-1}_{\infty, \infty}$, is smooth.
\end{theorem}
Using localized energy inequalities coupled with interpolation of $L^4$ between
$\dot{B}^{-1}_{\infty, \infty}$ and the homogeneous Sobolev space $\dot{H}^1(\bR^3)$, they prove the following. If $v$ is a suitable weak solution to the three dimensional NS from the space $L^\infty(0,T; \dot{B}^{-1}_{\infty, \infty})$, then a number of scaled energy quantities of $v$ are bounded. Consequently only type I blow up can occur. In the axially symmetric case, this has been ruled out in \cite{Seregin2012}. Therefore these solutions are smooth. We mention that although $BMO^{-1} \subset \dot{B}^{-1}_{\infty, \infty}$, the result in \cite{LZ11} goes beyond suitable weak solutions since there is no need for the local energy inequality.

Without knowing if blow up happens in general, it is desirable to find an upper
bound for the growth of velocity. It is expected that the solutions are
smooth away from the axis, with  certain  growing bound
when approaching the axis. The next theorem confirms this
intuitive idea. Although it did not give the bound
(\ref{v<1/r}) which is required for smoothness, it
reveals the exact gap between what we have and what we need.

 This seems to be the first pointwise bound for the speed
(velocity) for the axially symmetric Navier-Stokes equation.
  We mention that
a less accurate a priori upper bound for the vorticity has been found in Burke-Zhang \cite{BuZh} a few years earlier. Also,  in the paper  \cite{LZ11-2}, Lei-Zhang proved that if the scaling invariant quantity $r |v(x, t)|$ is sufficiently large at a point $(x_0, t_0)$, then the solution is close, in $C^{2, 1}$ sense, to a nonzero constant vector after a suitable scaling.

\begin{theorem} (Lei-Navas-Zhang \cite{LNZ})
 \label{a priori bound}  (velocity bound).
 Suppose  $v$ is a smooth,
axially symmetric solution of the three-dimensional Navier-Stokes
equations in ${\bR}^3\times(-T,0)$ with initial data
$v_0=v(\cdot,-T)\in L^2({\bR}^3)$. Assume further
$rv_0^{\theta}\in L^{\infty}({\bR}^3)$ and let $R= \min\{1,
\sqrt{T/2} \}$.

 Then for all $(x, t) \in {\bR}^3 \times (-R^2, 0)$, it holds

\[
|v^r(x, t)| + |v^3(x, t)| \le \frac{C \sqrt{|\ln r|}}{r^2},  \qquad 0< r
\le \min \{1/2, R \}.
\]Here  $r$ is the distance
from $x$ to the $z$ axis, and $C$ is a constant depending only on the initial data.

\end{theorem}

In the same paper, the following results are also proven.

\begin{theorem}
Under the same assumption as the previous theorem, there exists a constant $C$,
depending only on the initial data, such that,
\[
|L^\theta(x, t)| \le \frac{C |\ln r |^{1/2}}{r^{1/2}}, \qquad r \le \min \{1/2, R
\}.
\]
\end{theorem}

Here $L^\theta$ is the angular part of the stream function, which gives rise to $v^r$ and
$v^3$ by the following relations
\[
v^3 = \frac{1}{r} \partial_r ( r L^\theta), \quad v^r ={\blue -} \partial_{x_3} L^\theta.
\]So it is the most important component of the stream function (vector).

\begin{theorem}
\lab{thasnsvz}
Let $v$ be a Leray-Hopf solution to (\ref{eqasns}) in ${\bR}^3 \times (0, \infty)$ such
that
$r v^{\theta}(\cdot, 0) \in L^\infty({\bR}^3)$.
Suppose, for a given constant $C>0$,
and all $x \in {\bR}^3$ and $t \ge 0$,
\[
|v^3(x, t)| \le \frac{C}{r}.
\]Then $v$ is regular for all time.
\proof
\end{theorem}
From the relation $ v^3 = \frac{1}{r} \partial_r (
r L^\theta)$, using the notation $|x'|=r=\sqrt{x^2_1+x^2_2}$, we have,
\[
\bali
\left| |x'| L^\theta(x, t) \right|
&= \left| \int^{|x'|}_0 \partial_r ( r L^\theta) dr  \right| \le
 \int^{|x'|}_0 \left| r v^3 \right|  dr \le C |x'|.
\eali
\]Here we just used the  assumption on $v^3$.  Therefore $L^\theta$ is a bounded function.
Then from the main result in \cite{LZ11}, we know that $v$ is regular for all time.
\qed

Comparing with the previously mentioned results of Chen, Strain, Tsai and Yau\cite{CSTY2} and Koch, Nadirashvilli, Sverak and Seregin \cite{KNSS},  there is no
restriction
on $v^r$ in our case. See also the paper by Chen, Fang, T.Zhang \cite{CFZ}.

There are also regularity condition on one component of the
velocity and/or vorticity.
 J. Neustupa  and
M. Pokorny \cite{NP} proved that the regularity of one component
(either $v^r$ or $v^{\theta}$) implies regularity of the other
components of the solution. See more refined results in  P. Zhang and T. Zhang \cite{ZZ}. Also proving regularity is the work of
Q. Jiu  and  Z. Xin \cite{JX} under an assumption of sufficiently small
zero-dimension scaled norms.  D. Chae  and  J. Lee \cite{CL}  also proved
regularity results assuming finiteness of another
zero-dimensional integral. A pointwise critical blow up criterion: $|\o^\th|\leq \f{C}{r^2}$ was also given in Z. Li and X. Pan \cite{LP:2019CPAA}.

As mentioned earlier, X. H. Pan \cite{Panx} recently  obtained a $loglog$ improvement of the main
result in Chen-Strain-Tsai-Yau and Koch, Nadirashvili,  Seregin, and V. Sverak.  Although it looks like a small improvement, it is a slightly super critical result based on the argument in
\cite{LZ11}. The proof replies on the robustness of De Giorgi-Nash-Moser method to prove the function
$\Gamma = r v^\th$ has a modulus of continuity at the $z$ axis, under the slightly supercritical condition on $b=v^r e_r + v^3 e_3$. The vector $b$ controls the drift term in the equation for $\Gamma$: \eqref{eqvth}.

There are also global regularity result in special cases.
G. Tian  and  Z.
Xin \cite{TX}  constructed a family of singular axially
symmetric solutions with singular initial data.
T. Hou  and  C. Li \cite{HL} found a special class of global smooth
solutions. See also a recent extension: T. Hou, Z. Lei  and  C. Li
\cite{HLL}.

\subsection{Criticality of ASNS and closing of the scaling gap}

Despite these efforts, there is still a finite scaling gap between the regularity
condition and
 a priori bounds.   In almost all the literature, the regularity conditions
are critical and hence are scaling invariant under
standard scaling. Improvements are at most logarithmic in nature, except for resorting to further requirements such as sparseness of the sets where the vorticity is high. We have mentioned the Ladyzhenskaya-Prodi-Serrin condition
for regularity requires
the velocity to be bounded in suitable function space whose norm is invariant under
standard scaling, such as $L^{p, q}$ with $\frac{3}{p}+\frac{2}{q}=1$.
However the energy bound scales as
$-1/2$. So there is a finite gap which makes the equation supercritical.

However in a recent paper \cite{LZ17},  Lei and Zhang made the following observation.

{\it The vortex stretching term of the ASNS is critical.}

Previously it was believed to be super-critical, which means in micro scales the equation becomes chaotic and intractable by current method. Critical equations are still very difficult but more tools are available to study them.  In the next few pages we describe the result in more details.

Let $\o=\nabla \times v=\o^r e_r + \o^\theta e_\theta + \o^3 e_3$ be the vorticity.
Define
\be
\lab{defJO}
J= \frac{\omega^r}{r},  \quad \Omega=\frac{\omega^{\theta}}{r}.
\ee Then the triple $J, \O, \o^3$
 satisfy the system: for $b=v^r e_r + v^3 e_3$,
\begin{equation}
\label{eqjoo}
\begin{cases}
\Delta J  -(b\cdot\nabla) J +\frac{2}{r}\p_r J +
 (\o^r \p_r + \o^3 \p_{x_3}) \frac{v^r}{r} - \p_t J
=0,\\
\Delta \Omega -(b\cdot\nabla)\Omega+\frac{2}{r}\p_r
\Omega - \frac{2v^{\theta}}{r} J -\p_t \Omega=0,\\
\Delta \o^3-(b\cdot\nabla) \o^3+ \o^{r}\p_r v^3 + \o^3 \p_{x_3}
v^3 -\p_t \o^3=0.
\end{cases}
\end{equation}
These follow from direct computation based on the vorticity equation
\begin{align}
\lab{eqvort}
\begin{cases}
  \big (\Delta-\frac{1}{r^2} \big
)\omega^r-(b\cdot\nabla)\omega^r+\omega^r
\p_r v^r +\omega^3\p_{x_3} v^r -\p_t
\omega^r
=0,\\
   \big  (\Delta-\frac{1}{r^2}  \big
)\omega^{\theta}-(b\cdot\nabla)\omega^{\theta}+2\frac{v^{\theta}}
{r}\p_{x_3} v^{\theta}+\omega^{\theta}\frac{v^r}{r}-\p_t
\omega^{\theta}=0,\\
 \Delta\omega^3-(b\cdot\nabla)\omega^3+\omega^3\p_{x_3}
v^3+\omega^{r}\p_r v^3 -\p_t \omega^3=0,
\end{cases}
\end{align} and the relations
\be
\lab{w-v}
\o^r= -\p_{x_3} v^\theta, \quad \o^\theta= \p_{x_3} v^r-\p_r v^3, \quad \o^3 =
\p_r v^\theta +\frac{v^\theta}{r}.
\ee

We mention that the function $J$ was introduced in the  recent paper by H. Chen-D.Y.Fang-T. Zhang
\cite{CFZ}. By carrying out an energy estimate on the first two equations, they
proved the following result: if
\[
  |v^\theta(x, t)| \le \frac{C}{r^{1-\e}},
\] for all $x$ and $t>0$,
then solutions are regular everywhere.  Here $\e>0$ and $C$ are positive
constants.  This result gives a hint that the ASNS is a little super-critical.
The reason is that  $v^\theta$ has the well known
a priori bound
\be
\lab{vthjie}
|v^\theta(x, t) | \le \frac{1}{r}  \Vert r v^\theta(\cdot, 0) \Vert_\infty,
\ee which comes from equation \eqref{eqvth} via the maximum principle.

%Let us mention that \eqref{vthjie} follows from the equation for
%$\Gamma= r v^\theta$:
%\be
%\lab{eqvth}
%\Delta \Gamma - b \nabla \Gamma- \frac{2}{r} \p_r
%\Gamma-\p_t \Gamma=0,
%\ee and the maximum principle.

Now we  observe that the vortex stretching terms in all three equations in
\eqref{eqjoo} are critical
 when viewed in a suitable way. The key is to treat \eqref{eqjoo} as a closed system.
 Therefore the vorticity equation of
 3 dimensional axially symmetric Navier-Stokes equations are
critical instead of supercritical as commonly believed.

Here are the details.  From its a priori bound, we know that
$v^\theta$ at worst scales as $-1$ power of the distance. Using the relation
\eqref{w-v}, we see that $\o^r$ and $\o^3$ in
the vortex stretching terms in \eqref{eqjoo} at worst scale as $-2$.
The key observation is to treat $\o^r$ and $\o^3$ as potential functions
rather than unknowns.  It is well known that in a second order reaction diffusion equation, potentials
which scale as $-2$ power of the distance are critical instead of supercritical.

But how to treat the other terms $\p_r \frac{v^r}{r}$, $\p_{x_3} \frac{v^r}{r}$,
$\p_r v^3$ and $\p_{x_3} v^3$? It turns out that they can all be converted to
$J$, $\O$ and $\o^3$ which are treated as unknown functions in \eqref{eqjoo}.

In fact one has the following inequalities
\be
\lab{vr/rtoOm}
\Vert \nabla \frac{v^r}{r} \Vert_2 \le \Vert \O \Vert_2, \qquad
\Vert \nabla^2 \frac{v^r}{r} \Vert_2 \le \Vert \partial_3 \O \Vert_2.
\ee These can be seen from the identities
\[
v^r = -\p_{x_3} L_\theta,  \quad (\Delta + \frac{2}{r} \p_r ) \frac{L_\theta}{r}
=-\O
\]which imply
\[
(\Delta + \frac{2}{r} \p_r ) \frac{v^r}{r}
=\p_{x_3} \O.
\]Then one can use $\frac{v^r}{r}$ and $\Delta \frac{v^r}{r}$ as test
functions respectively to deduce \eqref{vr/rtoOm}.

Moreover from the relation
\[
\Delta \p_i v = - \nabla \times \p_i \o
\]and integration by parts, we know that
\[
\Vert \nabla \p_r v^3 \Vert_2 + \Vert \nabla \p_{x_3} v^3 \Vert_2
\le C \Vert \nabla \o \Vert_2.
\]By direct computation, we also have the pointwise relation
\[
|\nabla \o |^2 \le 2 r^2 (|\nabla J|^2 + |\nabla \O|^2) +
|\nabla \o^3|^2 + 2 (J^2 + \O^2).
\]
Therefore, even though $\o^r$ and $\o^3$ are viewed as potential
functions generated by $v^\theta$, the system \eqref{eqjoo} is still
a closed system of $J$, $\O$ and $\o^3$.

Note that in carrying out energy bound for equation \eqref{eqjoo} the drift terms (first order terms)
can be integrated out if functions decay sufficiently fast near infinity.  One can
also carry out a localized argument to take care of the drift term as in the paper
 \cite{Zcmp04}. We take the liberty to correct one misstatement in the text on p246 in that paper, where it was stated that weak solutions are Lipschitz in the spatial direction. It should have been "weak solutions are locally bounded".

Now let us introduce the main result in \cite{LZ17}.

\begin{definition}\label{FBC}
We say that the angular velocity $v^\theta(r, z, t)$ is in $(\delta_\ast, C_\ast)$-critical
class if
\begin{equation}\label{FBC-1}
\int \frac{|v^\theta|}{r}|f|^2dx \leq C_\ast\int|\partial_rf|^2dx + C_0\int_{r \geq
r_0}|f|^2dx,
\end{equation}
\begin{equation}\label{FBC-2}
\int |v^\theta|^2|f|^2dx \leq \delta_\ast\int|\partial_rf|^2dx + C_0\int_{r \geq
r_0}|f|^2dx,
\end{equation}
holds for some $r_0 > 0$, some $C_0 > 0$ and for all $t \geq 0$ and all axially symmetric
scalar and vector functions $f \in H^1$.
\end{definition}

Clearly, under the natural scaling of the Navier-Stokes equations:
$$v_\lambda(t, x) = \lambda v(\lambda^2t, \lambda x),\quad p_\lambda(t, x) = \lambda^2
p(\lambda^2t, \lambda x),$$
the above definition  is invariant:
$(v_\lambda)^\theta$ also satisfies \eqref{FBC-1}-\eqref{FBC-2}
 if $v^\theta$ does.

\begin{theorem}\label{thm2}
For arbitrary $C_\ast > 1$,  there exists a constant $\delta_\ast > 0$ such that
for all Leray-Hopf solutions $v$ to the axially symmetric Navier-Stokes equations with initial
 value $v_0$ satisfying $\|v_0\|_{H^{\frac{1}{2}}} < \infty$ and $\|r v^\th_0\|_{L^\infty} < \infty$, the
 following conclusion is true.
 If the angular velocity field $v^\theta$ is in $(\delta_\ast, C_\ast)$-critical class,
 i.e. $v^\theta$ satisfies the critical Form Boundedness Condition in
 \eqref{FBC-1}-\eqref{FBC-2},
then $v$ is regular globally in time.
\end{theorem}

An immediate corollary of the theorem is:
\begin{corollary}\label{cor}
Let $\delta_0 \in (0, \frac{1}{2})$ and $C_1 > 1$. Let $v$ be a Leray-Hopf solution to
the axially symmetric Navier-Stokes equations with initial data $v_0  \in H^{1/2}$
and $\|r v^\th_0\|_{L^\infty} < \infty$.
If
\begin{equation}\label{CD}
\sup_{0 \leq t < T}|r v^\theta (r, x_3, t)| \leq C_1|\ln r|^{- 2},\ \ r \leq \delta_0,
\end{equation}
then $v$ is regular globally in time.
\end{corollary}

We mention that $C_\ast$ in the theorem  and $C_1$ in the Corollary
\ref{cor} are independent of neither the profile nor the norm of the given initial data.
 The point is that if \eqref{CD} is satisfied,
then \eqref{FBC-1}-\eqref{FBC-2} is true. Therefore one can apply the
 theorem  to get the desired conclusion.

After \cite{LZ17} was posted on the arxiv, in the paper by Dongyi Wei \cite{Weid}, the power in the log term has
been improved
to $-3/2$. Namely, he proved
\begin{theorem} (\cite{Weid})
\lab{thWeid}
Let $v$ be a Leray-Hopf solution to
the axially symmetric Navier-Stokes equations with initial data $v_0  \in H^{2}$
and $\|r v^\th_0\|_{L^\infty} < \infty$. If, for some $\delta_0 \in (0, 1/2)$,
\begin{equation}\label{CDwei}
\sup_{0 \leq t < T}|r v^\theta (r, x_3, t)| \leq C_1|\ln r|^{- 3/2},\ \ r \leq \delta_0,
\end{equation} then $v$ is regular.
\end{theorem}

The improvement is achieved by decomposing the space time in a dynamic way when carrying out the energy estimate for the system of $\O$ and $J$ in \eqref{eqjoo}. More specifically, in \eqref{eqjoo}, one multiplies the first equation by $J$ and the second one by $\O$ and integrate in space time. One can justify the integration by using a cut off function $\phi^2$ before a potential singular time $t$. After integration by parts, one deduces
\be
\lab{jowz}
\aligned
&\frac{1}{2} \int \left(J^2+ \O^2 \right) \phi^2 dy \bigg |^t_0
+   \int^t_0 \int \left( | \nabla J|^2 +|\nabla \O|^2 \right) \phi^2 dyds\\
&\le  \underbrace{- \int^t_0 \int \frac{ 2 v^\theta}{r} J \O \phi^2 dyds }_{T_1} + \underbrace{\int^t_0 \int   (\o^r \p_r \frac{v^r}{r} + \o^3 \p_{x_3} \frac{v^r}{r})
J \phi^2 dy ds}_{T_2} + \text{less singular terms}.
\endaligned
\ee If we can absorb $T_1$ and $T_2$ by the left hand side, then, we would know that
$\nabla J$ and $\nabla \O$ are $L^2_{loc}$ in space time. Since $J=\o^r/r$ and $\O=\o^\th/r$, we then know that $\nabla \o^\th$ and $\nabla \o^r$ are $L^2_{loc}$ in space time around the $x_3$ axis.
One can also argue, using the first term on the left hand side that $\o^\th$ and $\o^r$ are in the space $L^\infty(0, t; L^2_{loc})$. With these information, it is well known by Sobolev imbedding and bootstrapping that regularity of solutions follow.
The term $T_1$ is the most singular one on the right hand side. The term $T_2$, after using \eqref{w-v},
\eqref{vr/rtoOm} and integration by parts, can be shown to be logarithmically less singular than $T_1$.
So the main task is to control $T_1$. Using Cauchy-Schwarz inequality, it is sufficient to use the left hand side of \eqref{jowz} to bound the terms
\[
\int^t_0 \int \frac{ |v^\theta|}{r} J^2 \phi^2 dyds, \qquad \int^t_0 \int \frac{ |v^\theta|}{r} \O^2 \phi^2 dyds
\]Let us chose a positive function $r_1=r_1(t) \le \epsilon K(\epsilon) a(t)^{-1}$, where $K$ is some exponential function and $\epsilon$ is a small number to be chosen suitably; $a(t)= \Vert  r^{-1} \int^r_0 |v^\th(\rho, z, t)| d\rho \Vert_{L^\infty}$.
Splitting the spatial domain along $r=r_1(t)$ and using cut-off and integration by parts, one shows that
\be
\lab{jiej2}
 \int \frac{ |v^\theta|}{r} J^2 \phi^2 dy \le \epsilon^{-1/3} \int |\p_r J|^2 dy + C r^{-2}_1 (\Vert \Gamma \Vert_{L^\infty} + \epsilon^{-1/3}) \int_{r \ge r_1/2} J^2 dy,
\ee provided that $\Vert \Gamma \Vert_{L^\infty(r \le r_1)}<\epsilon<1$. The same bound holds when $J$ is replaced by $\O$. Now using the extra condition \eqref{CDwei}, one can substitute \eqref{jiej2} into \eqref{jowz} and turn it into an ordinary differential inequality. Then the claimed bound for $\o^\th$ and $\o^r$ in  $L^\infty(0, t; L^2_{loc})$ space follows, giving us regularity.

The appearance of the log term is due to the special property of the axially symmetric Hardy's inequality : for all $\psi \in C^2_0(\bR^2)$, there is one positive constant $C$ such that
\[
\int \int \frac{1}{r^2 |\ln r|^2} \psi^2(r, x_3) r drdx_3 \le C \int \int |\p_r \psi|^2(r, x_3) r drdx_3
\] It would be interesting if one can lower the power on $|\ln r|$ even further in the regularity criteria.
However the drift term, which is almost harmless in the vortex equation so far, is the main
obstacle. For instance, there is a dimension expansion trick in removing the log term in the Hardy inequality. However the drift term no longer has the divergence free structure viewing
in high dimensions.

\section{Ancient Solutions}

Next we talk about another common way to study the Navier-Stokes equations and many other nonlinear
equations: blow up analysis.

Let $v$ be a Leray-Hopf solution to the NS.
Suppose a singularity happens in finite time $T$, we would like to know what is it?  So we blow down the solution $v$ or blow up the space time
near maximal points of $|v|$ in the time interval $[0, t_i] \subset [0, T)$ where $t_i$ is a sequence times approaching the singular time (like using a microscope). More specifically, let $\lambda_i =
\sup_{t \in [0, t_i]} |v|$ and pick points $(x_i, s_i)$ with $s_i \le t_i$ such that
$
|v(x_i, s_i)| \ge \lambda_i/2.
$ Consider the sequence of functions
\[
v_i(x, t) \equiv \lambda^{-1}_i v(\lam^{-1}_i x + x_i, \lam^{-2}_i t+t_i), \qquad P_i(x, t) \equiv \lam^{-2}_i P(\lam^{-1}_i x + x_i, \lam^{-2}_i t+t_i).
\] They are bounded solutions of the NS in a increasingly larger time interval. By standard regularity theory, $v_i$ sub-converges in $C^{2, 1}_{loc}$ topology to a limit function $v_\infty$.
The resulting function is still
a solution of NS. But it is a bounded solution existing on the time interval $(-\infty, 0)$.
We call such solutions ancient solutions. In general there is no reason that ancient solutions are bounded, even when they arise as blow up limits from possible finite time singularities of Leray-Hopf solutions. Of course we do not know if such singularities exist for NS. But this is still the case for other parabolic equations where singularities occur in finite time. The reason is that one can choose a different set of blow up points $(z_i, s_i)$ where $|v(z_i, s_i)|$ is large but is not comparable to the maximum of $|v|$ before $s_i$. However, since they have existed for a long time, they must be special. In other words, ancient solutions are rigid.  Even if it turns out that no finite time blow up occurs, ancient solutions still serve as approximation of the behavior of solutions in regions of high velocity.

Can one classify all ancient solutions?

The answer is not so easy in general without further assumptions, even for positive solutions to the heat equation.
For example $v= e^{x+t}$ is a nonconstant, positive ancient solution to the 1 dimensional heat equation in $\bR^1$. It is also not hard to see that $v=(0, e^{x_1+t})$ is an ancient solution to the
2 dimensional NS. Note this example shows a difference between stationary and non-stationary Liouville property since it is well known that positive solutions to the Laplacian, namely, positive harmonic functions in $\bR^n$ are constants.
However there are also similarities between the two. It is well known that harmonic functions on $\bR^n$ of sublinear growth are constants. The same conclusion was proven for ancient solutions to the heat equation  in Souplet-Zhang \cite{SoZh} in 2006, which can be extended to some noncompact manifold cases. For the NS, there is an additional twist. For any harmonic function
$h=h(x)$ on $\bR^3$ and $a=a(t)$ a $C^1$ function of time, the function $v=a(t) \nabla h(x)$ is an ancient solution of NS. To rule out this kind of solutions, we usually consider the so-called mild solutions only.
For simplicity, we confine ourselves  to solutions with locally finite energy,  although the notation of mild solutions can be defined for other, more singular solutions.

\begin{definition}
\lab{demildso}
A function $v \in L^\infty_{loc}(0, T; L^2_{loc}(\bR^3)) \cap L^2_{loc}(0, T; H^1_{loc}(\bR^3))$ is called a mild solution to the 3 dimensional NS if
\[
v(x, t) = \int_{\bR^3} G(x, t, y)  v_0(y) dy + \int^t_0 \int_{\bR^3} K(x, t-s, y) v \nabla v(y, s) dyds,
\]where $G=G(x, t, y)$ is the standard heat kernel on $\bR^3$ and $K=K(x, t-s, y)$ is the Stokes heat kernel on $\bR^3$.
\end{definition}
By direct calculation, it can be shown that a mild solution in $L^\infty(0, T; \bR^3)$ is H\"older continuous in $\bR^3 \times [\delta, T]$ for any $\delta>0$. This fact is useful in proving convergence results involving bounded mild solutions.

See \cite{Solo} for an earlier treatise and Chapter 5 of \cite{Tsai} for a recent discussion of the Stokes heat kernel and mild solutions. The latter is often called the Oseen kernel \cite{Os}.
The explicit formula is $K(x, t, y)=K(x-y, t, 0) \equiv (K_{ij}(z, t, 0))$ with $z=x-y$, $i, j=1, 2, 3$ and
\be
\lab{stokeshe}
K_{ij}(z, t, 0)= G(z, t, 0) \delta_{ij} + \frac{1}{4 \pi} \p_{x_i} \p_{x_j} \int_{\bR^3} \frac{G(w, t,0)}{|z-w|} dw.
\ee

So, a more reasonable question to ask would be:

\be
\lab{queliu}
\text{ Are sublinear, mild ancient solutions of 3 dimensional NS constants?}
\ee

We comment that this question is quite challenging without further restriction on the growth or decay of the ancient solutions. See a brief discussion at the end of the section.
If the answer to the above question is yes for certain class of ancient solutions, then we say that the Liouville property or theorem holds for that class.

Of course for the NS, there is the result of \cite{KNSS}, which we have seen a little in the previous section. Their results can be summarized as the following theorem. Note that a different proof for statement (b) in the stationary case was found in \cite{KPR15} later.

\begin{theorem}(\cite{KNSS} Koch-Nadirashvili-Seregin-Sverak)

(a). If $v$ is a bounded, mild ancient solutions of the 2 dimensional NS, then $v$ is constant.

(b). If $v$ is a bounded, mild ancient solutions of the 3 dimensional  ASNS, then $v$ is constant if also $v^\th=0$.

(c). If $v$ is a bounded, mild ancient solutions of the 3 dimensional  ASNS, then $v$ is $0$ if also $|v(x)| \le C/r$.
\end{theorem}

In the same paper,  the authors also  made the following conjecture:

 {\it  The Liouville property is true for bounded, mild ancient solutions of the 3 dimensional  ASNS}.

 Next we describe further results from a recent paper
by Lei, Zhang and Zhao \cite{LZZ}.
It is proven, in 2D and the 3D axially symmetric swirl free case, that the  Liouville property holds for mild ancient solutions  if the velocity fields are sublinear with respect to the spatial variable and the vorticity fields satisfy certain decay condition (see Theorem \ref{2Dresult} and Theorem \ref{3Dnoswirl}). We remark that, unlike the Liouville theorems in \cite{KNSS}, there is no need for the condition that solutions are bounded. Moreover,  counterexamples are given when the velocity fields are linear with respect to the spatial variable. This shows that, under the condition that solutions are sublinear with respect to the spatial variable, the Liouville theorems are sharp.

The other main result in that paper is a Liouville property, under an extra decaying assumption, for bounded ancient solutions of the ASNS with general nontrivial swirl (see Theorem \ref{3Dswirl}). Let $v$ be the bounded ancient mild solutions of the axially symmetric Navier-Stokes equations with $v^\theta \neq 0$ and let $\Gamma=rv^\theta$. We prove that if $\Gamma\in L^\infty_tL^p_x$ where $1\leq p<\infty$, then $v$ must be constants.

Actually, in the 3D axially symmetric case, on the above conjecture, one can add the extra condition that $\Gamma\in L^\infty_tL^\infty_x$ without losing much generality. The reason is that  $\Gamma$ is scaling invariant and it also satisfies the maximum principle. So if the initial value of a solution satisfies the bound, then it will persist over time. Therefore, any ancient solution from blow up process will still satisfy this bound. As mentioned, when $v^\theta\neq 0$, the Liouville property was proved in \cite{KNSS} under the condition $|v|\leq \frac{C}{r}$. In comparison, in \cite{LZZ}, one only needs the condition on one component $v^\theta$ of the velocity $v$ while no additional conditions are added on the other two components. Moreover, even though they haven't totally proved the conjecture in \cite{KNSS}, the result can still be considered as a step forward in understanding the conjecture.   That is because the condition $\Gamma\in L^\infty_tL^p_x$ with $p$ being any finite number seems not too far from $\Gamma \in L^\infty$.

 The following are the main results in \cite{LZZ}. We mention that solution $v$ in the next two theorems are also assumed to be locally bounded in space. Namely, if $x$ is in a compact set, then $|v(x, t)|$ is uniformly bounded for all $t$. This assumption was not stated in the corresponding theorems in \cite{LZZ} although it was implicitly stated in the text. See also a related result by Pan and Li \cite{PanLi} where $v$ is allowed to grow at $(-t)^{0.5^-}$ rate near $-\infty$.

\begin{theorem}
\label{2Dresult}
  Let $v$ be a smooth, locally bounded ancient solution of the 2D incompressible Navier-Stokes equations and let $\o =\nabla\times v$ be the vorticity. If  $
   \lim\limits_{|x|\rightarrow +\infty}|\o (x,t)|=0,$
  uniformly for all $t\in(-\infty,0)$, then $ \o \equiv 0$ and $v$ is harmonic.

  If, in addition, $v$ satisfies
  \begin{equation}\label{sublinear}
  \lim\limits_{|x|\rightarrow +\infty}|v(x,t)|/|x|=0,
 \end{equation}
 uniformly for all $t\in(-\infty,0)$, then $v$ must be a constant.
\end{theorem}

\begin{theorem}
\label{3Dnoswirl}
  Let $v$ be a smooth, locally bounded ancient solution of the 3D axially symmetric Navier-Stokes equations without swirl and let $\o=\nabla\times v=\o^{\theta}e_{\theta}$ be the vorticity. Define $\Omega=\frac{\o^\theta}{r}$, if \begin{equation}\label{w2}
   \lim\limits_{r\rightarrow +\infty}|\Omega|=0,
  \end{equation}
  uniformly for all $t\in(-\infty,0)$, then $\o^\theta\equiv 0$ and $v$ is harmonic.

  If, in addition, $v$ satisfies
  \begin{equation}\label{sublinear2}
   \lim\limits_{|x|\rightarrow +\infty} |v(x,t)|/|x|=0,
  \end{equation}
  uniformly for all $t\in(-\infty,0)$, then $v$ must be a constant.
\end{theorem}

\begin{theorem}\label{3Dswirl}
  Let $v$ be a bounded ancient mild solution of the 3D axially symmetric Navier-Stokes equations with $v^\theta\neq 0$ and let $\Gamma=rv^\theta$. If $\Gamma\in L^p_x L^\infty_t(\mathbb{R}^3\times(-\infty,0))$ where $1\leq p<\infty$, then $v$ must be a constant.
\end{theorem}

The condition $v$ being sublinear (the condition \eqref{sublinear}) in Theorem \ref{2Dresult} can not be removed, even when $\o \equiv 0$.  Hence, the above 2 dimensional Liouville theorem is sharp.
Here is a counterexample. Let
$$v=(x_1,-x_2),\qquad p=-\frac{1}{2}x_1^2-\frac{1}{2}x_2^2,$$
then $\o=\partial_1u_2-\partial_2u_1=0$, and $(v,p)$ satisfies
 the 2D  stationary Navier-Stokes equations. However, $v$ is not a constant solution.

The first conclusion in  Theorem \ref{3Dnoswirl},  the axially symmetric, swirl free case, shows that if
$\o^\theta$ is sublinear with respect to $r$, then $\o^\theta \equiv 0$.
 Here is a counterexample to show that the conclusion will be wrong if $\o^\theta$ is linear with respect to $r$, and consequently one cannot prove $v$ is harmonic. This infers that condition \eqref{w2} is also important. Let
$$v=(-x_1x_3,-x_2x_3,x_3^2), \qquad p=-\frac 1 2 x_3^4+2x_3,$$
then $v^\theta=v\cdot e_\theta=0$ and $(v,p)$ satisfies  the stationary ASNS
However $\o^\theta=-r\not\equiv0$, $\Delta v=(0,0,2)\neq0$.

In addition, the condition $v$ being sublinear to $x$ ( condition \eqref{sublinear2}) is also necessary, even when $\o^\theta=0$. Moreover, one can give a counterexample to show that if $v$ is linear in the spatial variable, then there exists a nontrivial ancient solution of ASNS without swirl. It then follows that Theorem \ref{3Dnoswirl} is sharp. For example, let
$$v=(-\frac{1}{2}x_1,-\frac{1}{2}x_2,x_3), \qquad p=\frac{1}{8}x_1^2+\frac{1}{8}x_2^2+\frac{1}{2}x_3^2,$$
then we have
$$v^{\theta}=v\cdot e_{\theta}=0,\qquad v^{r}=v\cdot e_{r}=-\frac{1}{2}r, \qquad v^3=x_3.$$

These imply that $\o^{\theta}=\p_{x_3} v^r - \p_r v^3=0$ and $(v, p)$ satisfies
 the stationary ASNS equations without swirl. However, $v$ is not a constant solution.

So the remaining case for the Liouville property, which is also the most difficult one, is
when $\Gamma= r v^\theta(x, t)$ does not decay near infinity. There are some partial results in the paper \cite{LRZ} by Lei-Ren-Zhang.

\begin{theorem}
\label{thzhouqi}
Let $v=v^\theta e_\theta +v^r e_r + v^3 e_3 $ be a bounded mild ancient solution to the ASNS such that $\Gamma= r v^\theta$ is bounded. Suppose  $v$ is periodic in the $x_3$ variable. Then $v= c e_3$ where $c$ is a constant.
\end{theorem}

Let us describe the general idea of the proof of this theorem.  One will prove, by the De Giorgi-Nash-Moser method that $\Gamma$ satisfies a partially scaling invariant H\"older estimate which forces $\Gamma \equiv 0$.
Then the problem is reduced to the swirl free case that is solved in \cite{KNSS}.
In general this method will break down in large scale, unless one imposes scaling invariant decay conditions on $v^r$ and $v^3$.  Although no decay conditions on $v^r$ or $v^3$ are assumed in the theorem, one can demonstrate that the classical Nash-Moser iteration method can be carefully adapted to this situation. The key observation is the following: the incompressibility condition $\nabla \cdot b=0$ with $b=v^r e_r + v^3 e_3$, along with the periodicity in $x_3$ gives one extra information on $v^r$. In fact one will essentially use  $v^r(r,\theta, x_3) = - \partial_{x_3} (L^\theta(r,\theta, x_3)-L^\theta(r,\theta,0)) \in (L^\infty)^{-1}$, where $L^\theta$ is the angular stream function. Another helpful factor is that the spatial domain
$\bR^2 \times S^1$ behaves like a 2 dimensional Euclidean space in large scale, even though it really behaves 3 dimensionally near the axis. This shows that $v^r$ scales by the critical order $-1$, which is quite helpful. But the same can not be said for $v^3$. Fortunately,  the role of $v^3$ is not as important as
$v^r$ in the $x_3$ periodic case.

In contrast to the absence of nontrivial partially periodic ancient solutions for ASNS, in another 3 dimensional parabolic flow, the Ricci flow, such a solution exists and represents a typical singularity:
$(\mathbb{S}^2 \times \bR) \times (-\infty, 0)$. This important fact is proven by Perelman \cite{Per}.

Next, we present a theorem  that deals with non-periodic case, under an extra condition that $\Gamma$ converges to its maximum at certain speed. Even though the result cannot yet reach  the full conjecture in \cite{KNSS}, its proof utilizes a method of constructing a weight function by solving an adapted PDE, which is then used in an energy estimate. It may be of independent value and use elsewhere. This result and the ones in \cite{LRZ} all were posted
in the preprint \cite{LRZ0}. In a review process, it was suggested by a reviewer to split that paper.

Let
\be
\lab{limsupgam}
\lim\sup_{r \to \infty} \Gamma = \lim\sup_{r \to \infty} \sup_{x_3, t} \Gamma(r, x_3, t).
\ee  It can be shown that if $v$ is any bounded ancient solution such that $\Gamma$ is bounded, then $\lim\sup_{r \to \infty} \Gamma = \sup \Gamma$.

\begin{theorem} (\cite{LRZ2})
\lab{thgudai2} Let $v=v^\theta e_\theta +v^r e_r + v^3 e_3 $ be a bounded mild ancient solution to the ASNS such that $\Gamma= r v^\theta$ is bounded. There exists a small number $\e_0 \in (0, 1)$, depending only on $\Vert v \Vert_\infty$,
 such that if
\be
\lab{giginfty2}
| \Gamma^2(r, x_3, t) - \lim\sup_{r \to \infty} \Gamma^2| \le \frac{\e_0}{r} \lim\sup_{r \to \infty} \Gamma^2
\ee
holds uniformly for $x_3$, $t$, and large $r$, then $v= c e_3$ where $c$ is a constant.
\end{theorem}

Let us mention that to study the equation of $\Gamma$ in an isolated manner is likely to fail. For example
without the divergence free condition on the vector field $b$. There is no Liouville property for the equation
\[
\Delta f - \frac{2}{r} \p_r f - {\blue (b \cdot \nabla)} f=0, f(0)=0,\,  \lim_{r \to \infty} f =c  \quad \text{in} \quad \bR^3,
\]even when $f=f(r)$ is a one variable function. One can just take $b=a(r) e_r$ with $a(r) \le 0$, $ r \exp(\int a(r) dr)$ being integrable on $[0, \infty)$.  Then solve the ODE
$f'' - (\frac{1}{r} + a(r) ) f'=0$ .

We remark that Question \eqref{queliu} is very difficult even in the bounded, axially symmetric case, especially without the extra assumption $\Gamma = r v^\th \in L^\infty$. The reason is that it contains another long standing question for the stationary NS as a special case. Namely, is a homogeneous D solution zero?  We will discuss the latter in detail in the following section.

At the end of this section, we introduce a local representation formula for smooth solutions of NS, which
is a useful tool in proving convergence results in NS during and after blow up or scaling, including ancient solutions.
It is largely due to {\blue O'Leary} \cite{Ole} and has been used in several papers including \cite{Zhang2006}.
However, as pointed out by H. J. Dong, there is a missing term in that formula. A corrected one is given
in the addendum of \cite{Zhang2006}, which is given below. As one application, we  can prove that for a local smooth solution $v$ of NS, $v$ and $\nabla v$ is controlled by a local integral of $v$ or $|\nabla v|$ without any assumption on the pressure. The smoothness assumption can be relaxed. We refer to \cite{Zhang2006} addendum for details.

Given $r>0$ and $(x, t)$ in the space time. Let $Q_r(x, t) = \{(y, s) \, | \, |x-y|<r,  t-r^2<s<t \}$ be a
standard parabolic cube.
Denote by $\Gamma=\Gamma(x, y)=1/(4 \pi |x-y|)$ the Green's function on $\bR^3$,  $G=G(x, t, y)$ the standard heat kernel, $K=K(x, t, y)$ the Oseen kernel given in \eqref{stokeshe}.
Fixing $(x, t)$, we construct a standard
cut-off function $\eta$ such that $\eta(y, s)=1$ in $Q_{r/2}(x,
t)$, $\eta(y, s)=0$ outside of $Q_r(x, t)$, $0 \le \eta \le 1$. Denote by $K_j$ the $j-th$ column of the Oseen kernel. The following vector was introduced in \cite{Ole}.
\[
\bali
\Phi_j = \Phi_j(x, t-s, y) &= \eta(y, s) K_j(x, t-s, y) + \frac{1}{4
\pi} \nabla \eta(y, s) \times \int_{{\bf R}^3} \frac{ curl K_j(x,
t-s, z)}{|z-y|} dz\\
&\equiv \eta K_j(x, t-s, y) + \overrightarrow{Z}_j(x, t-s, y).
\eali
\]

\begin{proposition}[\cite{Zhang2006}, Addendum]
 Suppose $v=(v_1, v_2, v_3)$ is a smooth solution of the NS \eqref{nse} in $Q_r(x, t)$ and $K_j$ be the $j-th$ column of the Oseen kernel $K$. Then
\be
\lab{vgongshi}
\aligned
&v_j(x, t)
= \int_{Q_r(x, t)} v(y, s) \cdot \left[ K_j(x, t-s, y) (\Delta \eta +
\partial_s \eta)(y, s) +
2 \nabla \eta(y, s) \nabla_y K_j(x, t-s, y) \right]dyds \\
&\quad + \int_{Q_r(x, t)} v(y, s) \cdot \left[(\Delta_y \overrightarrow{Z}_j +
\partial_s \overrightarrow{Z}_j)(x, t-s, y)  +   v(y, s) \nabla_y \Phi_j(x, t-s, y) \right] dyds\\
&\qquad +\underbrace{ \int_{B_r(x)} \left[(\nabla \eta \cdot \nabla_y \Gamma) v_j -
\partial_{y_j} \Gamma  (\nabla \eta \cdot v)
-\partial_{y_j} \eta  (\nabla_y \Gamma \cdot v )
 \right](y, t) dy.}
\endaligned
\ee
\end{proposition} In the above, if $X, Y$ are two vector fields. Then $X \nabla Y \equiv \sum_j X_j \partial_j Y$. The last term in \eqref{vgongshi} was the missing one.
Since all integrands vanish on the lateral boundary of $Q_r(x, t)$, one can apply $\nabla_x$  freely into the integrals on the right sides of
\eqref{vgongshi}. Also it is not hard to see that the formula also works if one replaces $(x, t)$ by
any points in $Q_{r/2}(x, t)$ while keeping the cube $Q_r(x, t)$ unchanged.

Using the explicit formula above, one deduces

\begin{proposition}[mean value inequalities for (NS), \cite{Zhang2006}, Addendum]

 Let $v$ be a smooth solution of the NS \eqref{nse} in $Q_{2r}(x, t)$. Then there exists an absolute
constant $\lambda>0$ such that

(a). \[
\aligned
 |v(x, t)| &\le  \frac{\lambda}{r^5} \int_{Q_r(x,
t)-Q_{r/2}(x, t)} |v(y, s)| dyds + \frac{\lambda}{r^3} \int_{B_r(x)-B_{r/2}(x)} |v(y, t)| dy\\
&\qquad + \lambda \int_{Q_r(x, t)} K_1(x, t; y, s)  \
|v(y, s)|^2 dyds, \endaligned
\] where $K_1(x, t; y, s) =1/(|x-y|+\sqrt{t-s})^4$, if $t \ge s$ and $0$ if $t<s$.

(b).
\[
\aligned
&|\nabla v(x, t)| \le  \frac{\lambda}{r^5} \int_{Q_r(x,
t)-Q_{r/2}(x, t)} | \nabla v(y, s)|  dyds
 + \frac{\lambda}{r^6}
\int_{Q_r(x, t)-Q_{r/2}(x, t)} |v(y, s)| dyds \\
&\quad + \frac{\lambda}{r^4} \int_{B_r(x)-B_{r/2}(x)} |v(y, t)| dy
 + \lambda \int_{Q_r(x, t)} K_1(x, t; y, s) |v(y, s)| \
|\nabla v(y, s)| dyds. \endaligned
\]
\end{proposition}

Using an iteration, one can also remove the kernel function $K_1$ from the above mean value inequality, assuming $v$ is in a parabolic type Kato class. The latter contains the standard $L^p_xL^q_t$ regularity class alluded in Section 1 for $q<\infty$.

\begin{definition}[\cite{Zhang2006}]  A vector valued function $b=b(x,t)$ in $L^1_{loc}(\bR^{n+1})$ is in class $K_1$ if it satisfies the following condition:
\[
\lim_{h \to 0} \sup_{(x,t) \in \bR^{n+1}} \int^t_{t-h} \int_{\bR^n} [ K_1(x,t;y,s) +K_1(x,s;y,t-h)]|b(y,s)|dyds= 0.
\]
\end{definition}
The parabolic Kato norm of $b$ on a time interval $[t_1, t_2]$ and scale $h$ is
\[
B(b, t_1, t_2, h) \equiv \sup_{(x, t) \in \bR^n \times [t_1, t_2]} \int^t_{t-h} \int_{\bR^n} [ K_1(x,t;y,s) +K_1(x,s;y,t-h)]|b(y,s)|dyds.
\]

\begin{proposition}[ \cite{Zhang2006}, Addendum]
Let $v$ be a local solution of the NS \eqref{nse} in $Q_{4r}(x, t) \subset \bR^3 \times \bR$, satisfying the energy inequality \eqref{enineq} localized in $Q_{4r}(x, t)$.
Suppose also that $v |_{Q_{4r}(x, t)}$ is in class $K_1$. Then both $v$ and $|\nabla v|$ are bounded functions
in $Q_{2r}=Q_{2r}(x, t)$.

Moreover, for some positive constants $C=C(v)$ and $r_0$,
depending only on the size of the Kato norm of $v$ in the time interval $[t-16 r^2_0, t]$ with scale $16 r^2_0$ such that the following hold. When
$0<r<r_0$,
\[
\aligned
 |v(x, t)| &\le  \frac{C}{r^3} \sup_{s \in [t-(2 r)^2, t]} \int_{B_{r}(x)} |v(y,
s)| dy,\\
|\nabla v(x, t)| &\le \frac{C}{r^5} \Vert \nabla
v \Vert_{L^1(Q_{2r})}
 + \frac{C}{r^4} \sup_{s \in [t-(2 r)^2, t]} \int_{B_{r}(x)} |v(y,
s)| dy.
\endaligned
\]
\end{proposition}

\section{Stationary Solutions}

\subsection {D solutions}

In this section we discuss the stationary Navier-Stokes equations
\begin{equation}
\label{SNS}
\left\{
\begin{aligned}
&{\blue (v \cdot \nabla)} v+\nabla p-\Delta v=f, \qquad \text{in} \quad \O \subset \bR^3, \\
&\nabla\cdot v=0.
\end{aligned}
\right.
\end{equation} Here $\O$ is an open domain and $f$ is a forcing term. Boundary conditions will vary case by case.

  We will focus on the decay and vanishing properties of the so-called homogeneous D solutions to  \eqref{SNS} in certain unbounded domain $\O \subset \bR^3$ with various boundary conditions and requirements of the behavior of $v$  at infinity. Here the name ``D solutions'' arises from the condition that solutions $v$ have  finite Dirichlet integrals (energy)
\be\label{edcondition}
\int_{\O}|\nabla v(x)|^2 dx<+\infty.
\ee

\begin{definition}
\lab{deDS}
A smooth solution to \eqref{SNS} is called a D solution on $\O$ if \eqref{edcondition} holds. It is referred to as DS.

A homogenous D solution on $\O$ is a D solution $v$ on $\O$ such that $v=0$ on $\partial \O$, $f=0$ and
$|v(x)| \to 0$ as $|x| \to \infty$. It is referred to as HDS.

An axially symmetric homogenous D solution on $\O$ is referred to as ASHDS.
\end{definition}

If no confusion arises, we will drop the reference to the domain $\O$ from D solutions.

 Existence of D solutions with several boundary conditions were studied in the pioneer work of Leray \cite{Le1} (p24) by variational method.
The following uniqueness problem has been open since then:

  \centerline{\it Is a homogeneous D solution equal to  $0$ ?}

 This  is also part of the very difficult uniqueness problem for the steady Navier-Stokes equation, which states that if the right hand of $\eqref{SNS}_1$ is a nontrivial smooth function, decaying sufficiently fast, do we have the uniqueness of D solutions. For axially symmetric compact domains with a hole at the axis, the uniqueness fails, as pointed out by Yudovich \cite{Yud} over 50 years ago. However for general domains, the problem is wide open. Let us explain Yudovich's construction of 2 solutions.
 Let $\O$ be the above compact domain. For a divergence free vector field $A$, find a nontrivial solution
 $\phi$ of the following "eigenvalue problem" for vector fields.
\be
\label{eqYudo}
\left\{
\begin{aligned}
&\Delta \phi - {\blue (A \cdot \nabla)} \phi - {\blue (\phi\cdot \nabla)} A-\nabla \tau= 0, \qquad \text{in} \quad \O, \\
&\nabla\cdot \phi=0 \quad \text{in} \quad \O; \quad \phi =0, \quad \text{on} \quad \p \O.
\end{aligned}
\right.
\end{equation}
Then $A+\phi$ and $A-\phi$ will be two distinct solutions of the stationary NS with the same forcing term, the same boundary value, but different pressure terms. But there is a problem. How to find a nontrivial solution to \eqref{eqYudo}. Yudovich found a special $A$ so that \eqref{eqYudo} has a variational structure and found a nontrivial solution by solving a  constrained maximum problem. His choice of $A$ is
\[
A= - r^{-3} e_\theta
\]in the polar coordinates.

Now let us focus on the decay and vanishing problems for DS. In the 2 dimensions, the corresponding vanishing property in the full space case is solved by Gilbarg and Weinberger \cite{GW:1}, using the line integral method.

 \begin{theorem} (\cite{GW:1} Let velocity $v$ and pressure $p$ be a solution of the Navier-Stokes equation defined over the entire $\bR^2$ and assume $\int_{\bR^2}|\nabla v(x)|^2 dx<+\infty.$
Then $v$ and $p$ are constant.
\end{theorem}

They also proved a number of asymptotic properties of 2 dimensional D solutions in general. For example, they proved that the pressure $p$ from Leray's D solution in 2 D exterior domains with $f=0$ must converge to a constant at infinity.

However, for the 3 dimensional problem, it is not even known if a general D-solution has any definite decay rate comparing with the distance function near infinity, even when the domain is $\bR^3$. The situation is akin to the regularity problem.  This time the scaling gap happens at infinity. Namely, to prove vanishing, one always needs to impose a condition that the solution decays
in some sense near infinity sufficiently fast. However, no decay rate has been proven a priori for HDS (homogeneous D solutions) in $\bR^3$, making the situation looks bleaker than the regularity problem. Another eerie similarity with the regularity problem
occurs in that some partial a priori decay can be proven quickly outside a small set.
In fact
R. Finn \cite{Finn} (page229), already observed the following partial decay property for any 3 dimensional vector field $v$  having a finite Dirichlet integral and tending to a constant vector $v_\i$ as $|x|\rightarrow \i$. For any $\dl>0$, there exists a measurable set $E_\dl\subset\bS^2$, such that $|E_\dl|\leq \dl$ and
\be\label{rf1}
|v-v_\i|\leq \f{C}{\dl^{1/2}|x|^{1/2}}, \q \ \forall x=|x|\o,\ \o\in {\bS^2} \backslash E_\dl.
\ee So the current situation is, in order to prove vanishing, one needs to establish a priori decay of sufficiently fast order, such as $|v(x)| \le c/|x|^{2/3}$, c.f. Theorem \ref{thKTW}. However the only a priori decay is a slow and partial one, except for solutions which are small in suitable sense.

Perhaps it is also not a surprise that just the decay and convergence property itself for general D solutions is an intricate matter in both 2 and 3 dimensions. In \cite{Finn59, Finn}, Finn proved that any D solution in 3 dimensional exterior domains converges pointwise uniformly  to a constant vector $v_\infty$ at infinity. Furthermore,  in  case $v_\infty \neq 0$, he showed that if $|v(x)-v_\infty| \le C|x|^{-\alpha}$ for some $\alpha >1/2$ as $x \to \infty$, then $\alpha$ can be replaced by $1$. He also introduced a concept called "physically reasonable (PR) solutions" to the stationary NS in 3 dimensional exterior domains, which are those satisfying $v(x)=O(|x|^{-1})$ if $v_\infty =0$ or $|v(x)-v_\infty| =O(|x|^{-\alpha})$ for some $\alpha >1/2$, if $v_\infty \neq 0$. Finn \cite{Finn65} then proved the existence and uniqueness of a PR solution in a 3 dimensional exterior domain when the boundary data are small enough. It is straight forward to show a PR solution is a D solution.  However, in case $v_\infty=0$, whether the converse is true has remained open till now. As mentioned, if $v_\infty=0$, we can not prove any decay rate at infinity for $v$ so far, not to mention the expected decay rate of $|x|^{-1}$. One exception is when the solution is small in a suitable sense. Then
the decay can be proven as a linear Stokes problem in exterior domains. See Galdi \cite{Gal93} where the viscosity constant $\nu$ is large and therefore a fixed Dirichlet integral is relatively small.
Note that one will not run into the Stokes paradox which happens when $v_\infty \neq 0$ and $v=0$ on the boundary of the exterior domain $\O$.
In the case of $v_\infty \neq 0$,  Babenko \cite{Bbk}  showed that every D solution is a PR solution if the forcing term is of bounded support.

One would be tempted to think that the 2 dimensional situation is simpler, which is the case for the regularity problem. But this is not true. Finiteness of the Dirichlet integral \eqref{edcondition} does not entail any decay of $v$ in 2 dimensions.  Only recently, Korobkov, Pileckas, and Russo \cite{KPR19} managed to prove the following convergence result.

\begin{theorem}
Let $v$ be a D-solution to NS in an exterior domain $\O \subset \bR^2$ without forcing term. Then $v$ converges pointwise uniformly at infinity to a constant vector $v_\infty \in \bR^2$.
\end{theorem}

The proof builds on earlier work in Gilbarg and Weinberger \cite{GW:1}, Amick \cite{Ami88, Ami91}. Fine analysis of the level sets of the vorticity plays an important role. These are curves in $\bR^2$.
There is still a problem on whether $v_\infty$ agrees with the one from Leray's original construction
\cite{Le1}. The interested reader may consult the introduction of \cite{KPR19} for a detailed account
of related and expanded results.

Let's  recall some of vanishing results with extra integral or decay assumptions for the solution $v$ in 3 dimensions.  If the domain $\O=\bR^3$, Galdi \cite{Galdi2011} Theorem X.9.5 proved that if $v \in L^{9/2}(\bR^3)$ is a homogeneous D-solution, then $v=0$. A log factor improvement was shown in Chae and Wolf \cite{Chae2016}. In \cite{Chae2014}, Chae proved that homogeneous D solutions are $0$ by assuming $\Delta v \in L^{6/5}(\bR^3)$, which scales the same as $\Vert \nabla v \Vert_2$.

\begin{theorem}(\cite{Chae2014})
Let $v$ be a homogeneous D solution in $\bR^3$. Suppose $\Vert \Delta v \Vert_{L^{6/5}(\bR^3)}<\infty$,  then $v=0$.
\end{theorem}

Chae's proof uses the property that the head pressure $Q=\frac{1}{2} |v|^2 + p$ is non-positive if $p(x) \to 0$ at infinity. This fact follows from the maximum principle. Then he uses $v (Q-\epsilon)^\delta_{-}$ as a test function in the stationary NS. Letting $\epsilon, \delta \to 0$, one concludes that $\int_{\bR^3} |\nabla v|^2 dx =0$, and hence $v=0$.

 Seregin \cite{Seregin2016} proved vanishing of homogeneous D solutions under the condition that $v \in L^6(\bR^3) \cap BMO^{-1}$. In a recent paper \cite{Kozono2017}, Kozono-Terasawa-Wakasugi showed vanishing of homogeneous D solutions if either the vorticity $\o=\o(x)$ decays faster than $c/|x|^{5/3}$ at infinity, or the velocity $v$ decays like $c/|x|^{2/3}$ with $c$ being a small number. i.e. They proved

\begin{theorem}(\cite{Kozono2017})
\lab{thKTW}
Let $v$ be a homogeneous D solution in $\bR^3$ and $\o =\nabla \times v$ be the vorticity.
If either $|\o(x)| = o( |x|^{-5/3})$ or $|v(x)| \le c |x|^{-2/3}$ for a small positive constant and all large $|x|$, then $v=0$.
\end{theorem}
Afterwards, W. D. Wang \cite{Wwdn} and N. Zhao \cite{Zhaona} proved a similar result for the axially symmetric D solutions independently. Note that no decay condition is imposed in the $e_3$ direction on one hand, but no improvement on the decay exponent is made on the other. However the result is still encouraging since we now have a priori estimates on $v$ and $\o$, c.f. Theorems \ref{thCJWe} and \ref{thfullwdecay}.

\begin{theorem}(\cite{Wwdn} and  \cite{Zhaona})
Let $v$ be an axially symmetric homogeneous D solution in $\bR^3$ and $\o =\nabla \times v$ be the vorticity.
If either $|\o(x)| = o( r^{-5/3})$ or $|v(x)| = o(r^{-2/3})$ for  all large $|x|$, then $v=0$.
\end{theorem}

The results in \cite{Kozono2017} has also been extended to D solutions in some Lebesgue and Morrey spaces in \cite{CJLR}.

 Under certain smallness assumption, vanishing result for homogeneous 3 dimensional solutions in a slab $\bR^2\times[0,1]$ was also obtained in the book \cite{Galdi2011}, Chapter XII.

As mentioned, D solutions in general noncompact domains have been studied in \cite{Le1}. Besides the whole space, the next simplest noncompact domains are the half space and slabs.
 Let us present a vanishing result for D solutions in a slab in $\bR^3$.

\begin{theorem}[\cite{CPZ2018-1} vanishing of HDS in a slab]
 \lab{ths2}
Let $v$ be a smooth solution to the problem
\be
\lab{3deqdiri}
\left\{
\begin{aligned}
&(v\cdot\nabla)v+\nabla p-\Delta v=0,  \quad \nabla\cdot v=0 \quad \text{in} \quad  \bR^2\times[0,1],\\
& v(x)|_{x_3=0}=v(x)|_{x_3=1}=0,
\end{aligned}
\rt.
\ee
such that the Dirichlet integral satisfies the condition:
\be\lab{d3d}
\int^1_{0}\int_{\bR^2}|\nabla v(x)|^2 dx<\infty.
\ee
Then, $v \equiv 0$.
\end{theorem}

 Comparing with the full space case, one can show by Poincar\'e inequality, with the help of Dirichlet boundary and the finite integral condition \eqref{d3d}, that the velocity $v$ belongs to $L^2$, which indicates that the decay rate of $v=v(x)$ is like $1/|x|$  in the integral sense. However one does not have a good knowledge of the pressure $p$. The Dirichlet boundary condition is known to induce complications on the vorticity and pressure. The main difficulty is to deal with the pressure term.
 Earlier, Pileckas and Specovius-Neugebauer \cite{PS:1}  studied the asymptotic decay of solutions of the Navier-Stokes equation in a slab. They prove, under certain weighted integral assumption on the velocity $v=v(x)$ and 3rd order derivatives, $v=v(x)$ decays like $1/|x|$ or $1/|x|^3$. See Theorem 3.1 in \cite{PS:1}. Then the vanishing of $v$ in the homogeneous case follows easily. However, these authors required that $(1+|x|)^{2+\beta} |v^3(x)|+(1+|x|)^{3+\beta} |\partial_{x_1} v^3(x)| $ with $\beta \in (-2, -1)$ is in $L^2$ in addition to further integral decay conditions of the first, second and third order derivatives of $v$, and consequently restriction on the pressure.  These conditions are not available to us. Furthermore
in the periodic case, which will be dealt with later, it is not even known that $v$ is $L^2$.

Besides if the Dirichlet energy is infinite, the vanishing property may be false. An example is
$v=(x_3(1-x_3), 0, 0), \, p=-2 x_1$.

In the paper \cite{CPZ2018-1}, there is an extra assumption that $v$ is bounded, which turns out to be
unnecessary.  Also it seems interesting  to study the decay property of nonhomogeneous solutions in a slab.

Now if the domain is the whole $\bR^3$, one can show that if the positive part of the radial component of D solutions decays at order $-1$ of the distance in spherical coordinates, then the D-solution vanishes.

\begin{theorem}[\cite{CPZ2018-1}]
\label{thgs}
Let $v_\rho=v_\rho(x)$ be the radial component of 3 dimensional D-solutions in spherical coordinates. If
\be\label{1.13}
{v_\rho(x)}\leq \f{C}{|x|}, \qquad x \in \bR^3,
\ee
for some positive constant $C$, then $v\equiv 0$.
\end{theorem}

 We should compare with the result in \cite{Kozono2017} where the authors prove, if  the weak $L^{9/2}$  norm of $v$ is small, then $v$ vanishes. This includes the case $|v(x)| \le c |x|^{-2/3}$ for certain small constant $c$. In contrast,  our assumption here is worse on the order of the distance function. However we only impose the condition on the positive part of the radial component of the solution and there is no restriction on the other two components.

 Next we concentrate on  axially symmetric homogeneous D solutions, for which the vanishing problem is also wide open. However as we shall see, a priori decay of the solutions $v$ and vorticity $\o$ in the $e_r$ direction is available, albeit insufficient for vanishing result in the whole space case. The set up of equations have been given in \eqref{eqasns} for the velocity and \eqref{eqvort} for the vorticity, except that we are happy to drop the time variable all together. The relation between $v$ and $\o$ are given in \eqref{w-v}.

 For the decay of $v$ and $\o$, the combined results of Chae-Jin  and S. K. Weng spanning 8 years can be condensed into:

 \begin{theorem} (\cite{CJ:1} and \cite{We:1})
 \lab{thCJWe} Let $v$ be a homogeneous D solution of ASNS in $\bR^3$.
 For $x \in \bR^3$,
\be
\lab{cjw}
|v(x)| \le C \left(\log r/r \right)^{1/2}, \q |\o^\th(x)| \le C r^{-(19/16)^-}, \q
|\o^r(x)| + |\o^3(x)| \le C r^{-(67/64)^-}.
\ee Here $C$ is a positive constant and for a positive number $a$, $a^-$ represents a number which is smaller than but close to $a$.
\end{theorem}

 Their proof is based on line integral techniques from Gilbarg and Weinberger \cite{GW:1}. In our recent work \cite{CPZ2018}, the decay estimate on the vorticity $\o$ is improved and a short proof for the decay of the velocity $v$ is found under a slightly more general condition.  Let us mention that in both the previous and following theorems, one can add a fast decaying forcing term and still obtain the decay estimates. The proof is the same.

\begin{theorem}[\cite{CPZ2018} a priori decay $v$ and $\o$]
\lab{thfullwdecay}

Let $u$ be a smooth axially symmetric solution to the problem
\be
\left\{
\begin{aligned}
&(v\cdot\nabla) v +\nabla p-\Delta v=0,\quad  \nabla\cdot v=0, \quad \text{in} \quad  \bR^3,\\
&\lim\limits_{|x|\rightarrow \i}v(x)=0,
\end{aligned}
\rt.
\ee
such that the Dirichlet integral satisfies the condition: for a constant $C$, and all $R \ge 1$,
\be\label{dr2r1}
\int^\i_{-\i}\int_{R \le |x'| \le 2 R}\big(|\nabla v(x)|^2 dx+ |v(x)|^6\big)dx'dx_3<C<\infty.
\ee
 Then the velocity and vorticity satisfy the following a priori bound. For a constant $C_0>0$, depending only on the constant $C$ in (\ref{dr2r1}) such that
 \[
 |v(x)| \le C_0 \frac{(\ln r)^{1/2}}{\sqrt{r}};
 \]
\[
\bali
&|\o^\th(x)| \le C_0  \f{(\ln r)^{3/4}}{r^{5/4}}, \quad |\o^r(x)|+|\o^3(x)| \le C \f{(\ln r)^{11/8}}{r^{9/8}}, \q r \ge e.
\eali
\]

\end{theorem}

Now we outline the proof of the decay result in Theorem \ref{thfullwdecay} briefly. We start with the observation that in a large dyadic ball, far away from the $x_3$ axis, after scaling, the axially symmetric Navier-Stokes equation resembles a 2 dimensional one. Then the 2 dimensional {\it Brezis-Gallouet} inequality introduced in \cite{BG:1} implies that a smooth vector field with finite Dirichlet energy is almost bounded. After returning to the original scale, one can show that $v$ is bounded by $C \big(\f{\ln r}{r}\big)^{1/2}$ for large $r$. It is curious that the ASNS hardly plays a role in the proof, except for guaranteeing $\Delta v$ is bounded.

 Next by combining the energy estimates of the equations for $\o$ in the stationary \eqref{eqvort} , {\it Brezis-Gallouet} inequality and scaling technique, we will show that, with $x_3$ taken as $0$ for convenience,
\be\label{evtheta}
|\o^\th(r,0)|\leq Cr^{-1}(\ln r)^{1/2}\|(v^r,v^\th,v^3)\|^{1/2}_{L^\i([\f{3}{4}r,\f{5}{4}r]\times[-r,r])},
\ee
and
\be\label{evrz}
\begin{aligned}
|\o^r(r,0)|+|{\blue \o^3} (r,0)|&\leq Cr^{-1}(\ln r)^{1/2}\|(v^r,v^3)\|^{1/2}_{L^\i([\f{3}{4}r,\f{5}{4}r]\times[-r,r])}\\
&\q\ +Cr^{-1/2}(\ln r)^{1/2}\|(\na v^r,\na v^3)\|^{1/2}_{L^\i([\f{3}{4}r,\f{5}{4}r]\times[-r,r])}.
\end{aligned}
\ee
The details can be founded  in \cite{CPZ2018}. Then using the decay of $v$ and \eqref{evtheta}, we can deduce that the decay rate of $\o^\th$ is $r^{-5/4}(\ln r)^{3/4}$.

In order to obtain decay of $\o^r$ and $\o^3$ from \eqref{evrz}, we need the decay of $\na v^r,\na v^3$ which can be connected with $\o^\th$ by the {\it Biot-Savart} law

\bes
-\Dl (v^re_r+v^3 e_3)=\na\times(\o^\th e_\th).
\ees
 Then $\na v^r,\na v^3$ can be written as integral representations of $\o^\th$ in the form of $\int_{\bR^3} K(x,y)\o^\th(y)dy$, where $K(x,y)$ are Calderon-Zygmund kernels. The decay relations between $\na v^r,\na v^3$ and $\o^\th$ are shown in Lemma 3.2 of \cite{CPZ2018}. At last, a combination of \eqref{evrz}, decay of $v$ and $\na v^r,\na v^3$ imply the decay of $\o^r,\o^3$ in Theorem \ref{thfullwdecay}.

 Clearly, if the Dirichlet integral is finite i.e. $\Vert \nabla v \Vert_{L^2(\bR^3)} < \infty$, then \eqref{dr2r1} is satisfied. In \cite{CPZ2018}, we also proved a vanishing result when D solutions are periodic in $x_3$ variable under the additional assumption that $v^\th$ and $v^3$ have zero mean in the $x_3$ direction. With some extra work, one can also reach the vanishing result assuming the integral in \eqref{dr2r1} grows at a certain positive  power of $R$.

In a subsequent paper \cite{CPZ2018-1}, the extra condition that $ v^\th, v^3$ have zero mean in the $x_3$ direction has been removed, under the stronger assumption that the Dirichlet integral is finite. We state it as the following theorem, with ASHDS standing for axially symmetric homogeneous D solutions.

%In that paper,  In a recent paper \cite{LRZ:1}, for solutions with infinite energy, Liouville property for  bounded,  axially symmetric solutions of the Navier-Stokes equation were studied under the natural assumption {\color{blue}$r u^\theta$} is bounded.
%Assuming in addition that $r u^\theta$ is bounded and $u$ is periodic in  $z$ variable, then it was shown that $u \equiv 0$. We mention that periodic solutions are also studied intensely in connection to the Kolmogorov flow.

%Earlier, the authors of the papers \cite{CSTY:1} and \cite{KNSS:1} proved Liouville theorems for ASNS under the assumption that $|u(x)| \le C/r$. See also an extension to $BMO^{-1}$ space in \cite{LZ:1}.

% The next result of this paper is an improvement of our main result in \cite{CPZ:1} in the case that the Dirichlet integral is finite. We will prove vanishing result when D-solutions are periodic in the third variable without any other assumption. Then we turn to axially symmetric solutions in a slab with Dirichlet boundary condition even if the Dirichlet integral has some growth. We will show that actually the  angular component of the solution vanishes.
%
% In the next theorem the flow is periodic in the $z$ direction with period $2 \pi$, a number chosen for convenience. Any other positive period also works. We will always take the forcing term $f=0$ throughout the paper.

\begin{theorem}(Vanishing of Periodic ASHDS)
\lab{thzp}
Let $v$ be a smooth axially symmetric solution to the problem
\be\label{1.7}
\left\{
\begin{aligned}
&(v\cdot\nabla)v +\nabla p-\Delta v=0, \quad \nabla\cdot v=0, \quad \text{in} \quad  \bR^2 \times \bS=\bR^2 \times [-\pi, \pi],\\
&v(x_1,x_2, x_3)=v(x_1,x_2,x_3+2\pi),\\
&\lim\limits_{|x'|\rightarrow \i} v(x)=0,
\end{aligned}
\rt.
\ee
with finite Dirichlet integral:
$
\int^\pi_{-\pi}\int_{\bR^2}|\nabla v(x)|^2dx<+\infty.
$
   Then $v=0$.
\end{theorem}
%\begin{remark}
%Assuming finiteness of the Dirichlet integral, Theorem\ref{thzp} is an improvement of our previous result Theorem 1.2 in \cite{CPZ:1} by removing the extra assumption that $u^\th$, $u^z$ have zero mean in the $z$ direction there. However,  the method in \cite{CPZ:1} is much different from the current one. Besides, the proof there can be applied to the case where the local Dirichlet integral has some growth, which means we can prove the vanishing result under the assumption $\int^\pi_{-\pi}\int_{|x'|\leq r}|\nabla u(x)|^2dx<(1+r)^\al$ for some suitable and positive $\al$. This method potentially allows application for flows with infinite Dirichlet energy such as Kolmogorov flows.
%\end{remark}

We outline the proof of the above results briefly. We start with the observation that in $x_3-$periodic case with $\O=\bR^2\times \bS$, the horizontal radial component of the solution $v^r$ satisfies $\int^\pi_{-\pi} v^r dx_3=0$. Poincar\'{e} inequality and the finite Dirichlet integral condition indicate that $v^r\in L^2(\O)$. Then using of $x_3$-periodicity again, we can actually prove that the $L^\i$ oscillation of the pressure $p$ is bounded in a dyadic annulus. The method is similar to the line integral technique in \cite{GW:1} except it is carried out in a 3 dimensional domain. At last by testing the vector equation \eqref{1.7} with $v \phi^2(|x'|)$, where $\phi(x')$ is supported in $\{x'|\, |x'|<2R\}$ and equal to $1$ in $\{x'| \, |x'|<R\}$, and making $R$ approach  $\i$, we can prove that $v\equiv 0$.
This result seems to add an extra weight in the belief that ASHDS in $\bR^3$ is $0$. The reason is that in the periodic case, finiteness of the Dirichlet integral of $v$ does not imply any decay of $v$. For example
the function $f= \ln \ln (2+r)$ has finite Dirichlet integral.  Yet $v$ still vanishes.

The next theorem treats the case with Dirichlet boundary condition (DBC) in a slab, even allowing the Dirichlet integral to  be log divergent. Recall from Theorem \ref{ths2} that 3 dimensional D solutions in a slab
with DBC is $0$. Why do we come back to the slab case? One reason is that infinite Dirichlet energy may induce non-uniqueness/non-vanishing as shown in the example after Theorem \ref{ths2}. So one would like to know,  under what rate of divergence of the energy,   vanishing of D solutions and especially ASHDS is preserved. In addition, solutions with infinite Dirichlet energy is also of interest in turbulence theory.
These include the study of Kolmogorov flows such as flows in a channel with periodic forcing terms. See \cite{Fri} e.g.

\begin{theorem} (\cite{CPZ2018-1} \lab{thsd})
Let ${\blue v}$ be a smooth, axially symmetric solution to the problem
\be\label{e1.8}
\left\{
\begin{aligned}
&({\blue v}\cdot\nabla)v+\nabla p-\Delta v=0, \quad \nabla\cdot v=0,\quad \text{in} \quad  \bR^2\times[0,1],\\
&\lim\limits_{|x'|\rightarrow \i} v=0,\q v(x)|_{x_3=0}=v(x)|_{x_3=1}=0,
\end{aligned}
\rt.
\ee
such that the Dirichlet integral satisfies the condition: for a constant $C$, and all $R \ge 1$,
\be
\lab{dr2r}
\int^1_{0}\int_{R \le |x'| \le 2 R}|\nabla v(x)|^2 dx<C<\infty.
\ee
Then $v^\th = 0$. Moreover, there exists a positive constant $C_0$, depending only on the constant $C$ in (\ref{dr2r})  such that
\be\label{edecay}
|v^r(x)| + |v^3(x)| \le C_0 \left(\f{\ln r}{r}\right)^{1/2}.
\ee
\end{theorem}

Since $v^\th = 0$ i.e. the flow is swirl free, the Navier-Stokes system reduces to
 \be\label{e1.9}
\lt\{
\begin{aligned}
&(v^r\p_r+v^3\p_{x_3})v^r+\p_r p=(\Dl-r^{-2})v^r,\\
&(v^r\p_r+v^3\p_{x_3})v^3+\p_{x_3} p=\Dl v^3,\\
&\p_rv^r+ r^{-1} v^r +\p_{x_3} v^3=0,\\
&\lim\limits_{r\rightarrow \i}(v^r,v^3)=0,\q (v^r,v^3)(r,x_3)|_{x_3=0,1}=0.
\end{aligned}
\rt.
\ee
\qed

So our vanishing problem now is much like a two dimensional problem. But we do not know any vanishing result for swirl free case in a slab with Dirichlet boundary condition and
(\ref{dr2r}).

The decay estimate in the $e_r$ direction still holds if there is an inhomogeneous term of sufficiently fast decay. However, we are not aware any decay estimate in the $e_3$ direction, except in the swirl free case, c.f. Theorem 1.2 \cite{We:1}.
If one works a little harder, then one can reach the same conclusion as the theorem assuming the integral in (\ref{dr2r}) grows at certain power of $R$.

Let us give a description of the proof of Theorem \ref{thsd}.   From \eqref{eqvth}, the quantity $\G:=rv^\th$ satisfies
 \be\label{1.9}
(v^r\p_r+v^3\p_{x_3})\G-(\Dl-\f{2}{r} \p_{r})\G=0.
\ee
So $\G$ enjoys maximum principle, which means, for bounded open sets $\O \subset \bR^3$,
\be\label{1.10}
\sup\limits_{x\in\O}|\G|\leq \sup\limits_{x\in\p\O}|\G|.
\ee We will show that the decay rate of $v^\th$ is at least $r^{-(\f{3}{2})^-}$ for large $r$. Take this decay for granted at the moment.
By using the above maximum principle and a sliding argument along the $x_3$ axis together with the fact that $\G(x) \to 0$ as $r \to \infty$, one can prove $v^\th\equiv 0$.

We are left to prove the above decay property of $v^\th$, which take three steps.

In step one: the Laplace Green's function $\mathcal{G}$ on $\bR^2\times [0,1]$ with homogeneous Dirichlet boundary condition will be introduced and a number of properties of $\mathcal{G}$ explained. The point is that $\mathcal{G}$ has fast decay near infinity. See \cite{MeMe:1} e.g.

In step two:  we obtain the following decay of $\o^r,\o^3$ by using a refined {\it Brezis-Gallouet} inequality, energy methods and scaling techniques:
\bes
|(\o^r(x),\o^3(x))|\ls r^{-1}\ln r,
\ees
 for large $r$. Furthermore, by the same procedure, we can show that
 \be
 |(\p_{x_3} \o^r(x),\p_r \o^3(x))|\ls r^{-3/2}(\ln r)^{3/2}.\\
 \ee

In step three, we use the {\it Biot-Savart} law to get the representation of $v^\th$ by integrals involving $\mathcal{G}$ and $(\p_{x_3} \o^r,\p_r \o^3)$ which implies that $v^\th$  decays in the same rate as $\p_{x_3} \o^r$ and $\p_r \o^3$. Now that we know $v^\theta$ decays faster than order $1$,
 as mentioned one can apply the maximum principle on the function $\Gamma = r v^\th$ to conclude $v^\th=0$.

When deriving the results in \cite{CPZ2018, CPZ2018-1}, one uses the standard tools such as energy estimate, vorticity equation and inter-play of velocity and vorticity. The new input is dimension reduction and the use of two
additional tools: Brezis-Gallouet inequality and Bogovskii's estimate, which are given a brief overview here. Some refinements of them are also given for a class of domains.

The Brezis-Gallouet inequality is a limiting case of the Sobolev inequality in 2 dimensions for $H^2$ extension domains. An analogous inequality also holds in higher dimensions by Brezis-Wainger \cite{BreWai}.  Recall that a domain $\O \subset \bR^n$ is called a $H^2$ extension domain if the
following properties hold.  There exists an extension operator $P: H^2(\O) \to H^2 (\bR^n )$ such that
$P$ is a bounded operator from $H^i( \O) $ to  $H^i (\bR^n)$, $i=1, 2$ and $ P f |_{\O}=f$ for all
$f \in H^2(\O)$. For instance, Lipschitz domains are $H^2$ extension domains. See \cite{Cald} e.g.

\begin{lemma}[\cite{BG:1}] \label{bg}
 Let $\O \subset \bR^2$ be a bounded open domain with $H^2$ extension property,  its complement or $\bR^2$. Let $f\in H^2(\O)$.  Then there exists a constant $C_{\O}$, depending only on $\O$, such that
\be
\label{BGinq}
\|f\|_{L^\i(\O)}\leq C_{\O}\|f\|_{H^1(\O)}\log^{1/2} \big(e+\f{\|\Dl f\|_{L^2(\O)}}{\|f\|_{H^1(\O)}}\big).
\ee
\end{lemma} The proof uses extension properties of $H^2$ functions and Fourier transform. A variant of it can be found in \cite{HL2} with a proof using the Green's function.

It is easy to see that the above inequality implies the next one, which was used in \cite{CPZ2018, CPZ2018-1}.
\be\label{BGinq1}
\|f\|_{L^\i(\O)}\leq C_{\O}(1+\|f\|_{H^1(\O)})\log^{1/2} \big(e+\|\Dl f\|_{L^2(\O)}\big).
\ee
The constant $C_{\O}$ depends on the domain in an implicit way. In applications, it is convenient to have an estimate on $C_{\O}$. Next we  prove the following refined $Brezis-Gallouet$ inequality for a class of domains in the $(r, x_3)$ plane, whose constant is independent of the thinness of the domains.
A price to pay is that the functions need to have zero boundary value in the $x_3$ direction or
mean zero in the $x_3$ direction.

\begin{lemma}
\label{leRBG}
For $R\gg 1$ and $0\leq\al\leq 1$, set
\bes
\bar{\mathcal{D}}_0=\big\{(r,x_3): 1-\f{1}{2}R^{\al-1}\leq r\leq1+ \f{1}{2}R^{\al-1},\ |x_3|\leq R^{\al-1}r^{\al}\big\}.
\ees
Then if $f \in H^2(\bar{\mathcal{D}}_0)$ satisfies
\be\label{bc}
f|_{|x_3|= R^{\al-1}r^{\al}}=0,
\ee
we have
\be\label{Bg3}
\|f\|_{L^\i(\bar{\mathcal{D}}_0)}\leq C_0(1+\|\na f\|_{L^2(\bar{\mathcal{D}}_0)})\log^{1/2}(e+R^{\al-1}\|\Dl f\|_{L^2(\bar{\mathcal{D}}_0)}),
\ee
where $C_0$ is independent of $R$. Here $\nabla=e_r \p_r  + e_3 \p_{x_3}$ and $\Delta=\p^2_r + \p^2_{x_3}$
are the two dimensional gradient and Laplacian in the $(r, x_3)$ plane, respectively.
\end{lemma}
\pf
 Note that we can not simply make zero extension for $f$ outside of the domain and apply the
 regular Brezis-Gallouet inequality. The reason is that the extended function may not be in $H^2$.

Define with scaled function $\t{f}(\t{r},\t{x_3})=f(R^{\al-1}\t{r},R^{\al-1}\t{x_3})$ where $(\t{r},\t{x_3})\in \t{\mathcal{U}}_{0}$ and
\bes
\t{\mathcal{U}}_{0}=\big\{(\t{r},\t{x_3}):R^{1-\al}-1/2<\t{r}<R^{1-\al}+1/2,\ |\t{x_3}|\leq |R^{\al-1}\t{r}|^{\al}\big\}.
\ees Observe that $\t{\mathcal{U}}_{0}$ is almost a square for large $R$.

Using \eqref{BGinq1}, we know
\be\label{Bg2}
\bali
&\|f(r,x_3)\|_{L^\i(\bar{\mathcal{D}}_{0})}
=\|\t{f}(\t{r},\t{x_3})\|_{L^\i(\t{\mathcal{U}}_{0})}\\
\leq& C_0(1+\|\t{f}\|_{H^1(\t{\mathcal{U}}_{0})})\log^{1/2} \big(e+\|\Dl \t{f}\|_{L^2(\t{\mathcal{U}}_{0})}\big)\\
=& C_0(1+\|\na\t{f}\|_{L^2(\t{\mathcal{U}}_{0})}+\|\t{f}\|_{L^2(\t{\mathcal{U}}_{0})})\log^{1/2} \big(e+\|\Dl \t{f}\|_{L^2(\t{\mathcal{U}}_{0})}\big).
\eali
\ee

 The point is that the constant $C_0$ is independent of $R$. This is because we can first extend the function $\t{f}$ to be a $H^2$ function in the whole $(\t{r}, \t{x_3})$ plane.  From the proof of the original
Brezis-Gallouet inequality, we know the constant relies only on the $H^2$ extension property of functions in a domain. The extension property only depends on the thickness of the original domain, which is scaled to $1$.

By the change of variables and relationship between $f$ and $\t{f}$, we deduce
\bes
\|\na\t{f}\|_{L^2(\t{\mathcal{U}}_{0})}=\|\na{f}\|_{L^2(\bar{\mathcal{D}}_{0})}, \|\t{f}\|_{L^2(\t{\mathcal{U}}_{0})}=R^{1-\al}\|{f}\|_{L^2(\bar{\mathcal{D}}_{0})},
\|\Dl \t{f}\|_{L^2(\t{\mathcal{U}}_{0})}=R^{\al-1}\|\Dl f \|_{L^2(\bar{\mathcal{D}}_{0})}.
\ees
Inserting the above equalities into \eqref{Bg2}, we find
\be\label{2.22}
\bali
&\|f(r,x_3)\|_{L^\i(\bar{\mathcal{D}}_0)}
\le C_0(1+\|\na{f}\|_{L^2(\bar{\mathcal{D}}_{0})}+R^{1-\al}\|{f}\|_{L^2(\bar{\mathcal{D}}_{0})})\log^{1/2} \big(e+R^{\al-1}\|\Dl f \|_{L^2(\bar{\mathcal{D}}_{0})}\big).
\eali
\ee

Now if $f$ satisfies \eqref{bc}, the Poincar\'e inequality implies
\bes
R^{1-\al}\|{f}\|_{L^2(\bar{\mathcal{D}}_{0})}\leq C\|\na{f}\|_{L^2(\bar{\mathcal{D}}_{0})},
\ees
where $C$ is independent of $R$. At last, combination of the above inequality and \eqref{2.22} infers \eqref{Bg3}. \qed

Now let us introduce Bogovski\u{i}'s \cite{Bme:1,Bme:2} work on solving the divergence equation on a bounded domain with $W^{1, p}_0$ functions \eqref{e4.3e}. We only present a special case written
 as Lemma III.3.1 of \cite{Galdi2011}. More general results can be found in Chapter III of  the same book.  See also a later paper by Brezis and Bourgain
 \cite{BrBo} for further results in the special case $p=n$ on torus.

\begin{lemma}
\label{L4.1L}
Let $\mO \subset \bR^n$, $n \ge 2$, be a bounded domain which is star shaped with respect to every point in a ball $B(x_0, R) \subset \mO$. Then for any $f \in L^2(\mO)$, satisfying
\bes
f \in L^p(\mO),  1<p<\infty, \qq {\rm with}\qq \int_{\mO} f=0,
\ees
  there exists a constant $C=C(\mO, p, n) $ and at least one vector field $V:\mO\rightarrow \bR^n$ such that
  \be\label{e4.3e}
\na\cdot V= f,\q V \in W^{1,p}_0(\mO),\qq \|\na V\|_{L^p} \leq C \|f\|_{L^p}.
\ee Furthermore, let $diam(\mO)$ be the diameter of $\mO$, there is a positive constant $c_0(n, p)$, depending only on $n, p$ such that  the following estimate holds:
\[
C \le c_0(n, p) \left[diam(\mO)/R\right]^n \left(1+ diam(\mO)/R \right).
\]
\end{lemma}

The proof is based on an explicit integral formula found by Bogovski\u{i}.
\be
\lab{Bogform}
V(x) =\int_\mO N(x, y) f(y) dy, \quad N(x, y) =\frac{x-y}{|x-y|^n}
\int^\infty_{|x-y|} \phi(y+  r \frac{x-y}{|x-y|}) r^{n-1} dr,
\ee where $\phi \in C^\infty_0(B_R(x_0))$ with $\int \phi =1$ is a fixed function.

 Sometimes the constant $C$ in \eqref{e4.3e} can be improved for some domains, as indicated in the  following:

\begin{proposition}\label{P4.1} Consider the domains $\mO_R=\big\{x \, | \,  R\leq |x_h|\leq 2R,\ |x_3|\leq |x_h|^\al\big\} \subset \bR^3$, where $\al\in [0,1]$ and $R\geq 1$.
For any
$
f\in L^2(\mO_R)$ with  $\int_{\mO_R} f=0
$, Problem \eqref{e4.3e} with $p=2$ has a solution such that
\be\label{e4.4e}
\|\na V\|_{L^2(\mO_R)}\leq  C_\al R^{1-\al}\|f\|_{L^2(\mO_R)},
\ee
where $C_\al$ is independent of $R$.
\end{proposition}
\pf The existence of $V$ is already known by Lemma \ref{L4.1L}. So we just need to prove (\ref{e4.4e}).
For $\bar{x}=(\bar{x}_1,\bar{x}_2,\bar{x}_3)\in \mO$, define
\[
\bar{f}(\bar{x}_1,\bar{x}_2,\bar{x}_3):=f(R\bar{x}_1,R\bar{x}_2,R^\al\bar{x}_3)= f(x_1, x_2, x_3).
\]Note $x_1=R\bar{x}_1$, $x_2=R\bar{x}_2$ but $x_3=R^\al\bar{x}_3$.

It is easy to see that $\bar{f}$ satisfies the assumption in Lemma \ref{L4.1L}. So by Lemma \ref{L4.1L}, there exists a vector function $\bar{V}:\mO\rightarrow \bR^3$ satisfying \eqref{e4.3e}. Then for $x\in\mO_R$, define
\be
\bali
&\q V(x_1,x_2,x_3)=(V^1(x_1,x_2,x_3),V^2(x_1,x_2,x_3),V^3(x_1,x_2,x_3))\\
&=(R\bar{V}^1(\f{x_1}{R},\f{x_2}{R},\f{x_3}{R^\al}),R\bar{V}^2(\f{x_1}{R},\f{x_2}{R},\f{x_3}{R^\al}),R^\al\bar{V}^3(\f{x_1}{R},\f{x_1}{R},\f{x_3}{R^\al}))\\
&=(R\bar{V}^1(\bar{x}_1,\bar{x}_2,\bar{x}_3),R\bar{V}^2(\bar{x}_1,\bar{x}_2,\bar{x}_3),R^\al\bar{V}^3(\bar{x}_1,\bar{x}_2,\bar{x}_3)).
\eali
\ee
By a direct computation, we have
\bes
\na\cdot {V}={f},\q {V}\in W^{1,2}_0(\mO_R),\qquad \text{in} \quad {x} \quad \text{variables}
\ees
\bes
\na\cdot {\bar V}= \bar{f},\q \bar{V}\in W^{1,2}_0(\mO), \qquad \text{in} \quad \bar{x} \quad \text{variables},
\ees
 where $\bar{V} = (\bar{V}^1(\bar{x}), \bar{V}^2(\bar{x}), \bar{V}^3(\bar{x}))$.
Now we  estimate the $L^2$ norm of $\na V$. We use $\alpha,\beta$ to take values only on $1,2$ and $i,j$ to take values on $1,2,3$. So we have
\bes
\bali
&\|\na V\|^2_{L^2(\mO_R)}
=\sum\limits^3_{i,j=1}\int_{|x_3|\leq |x_h|^\al}\int_{R\leq |x_h|\leq 2R}|\frac{\p V^j}{\p x_i}|^2dx_hdx_3\\
=&\int_{|x_3|\leq |x_h|^\al}\int_{R\leq |x_h|\leq 2R}\big(\sum\limits^2_{\alpha,\beta=1}| \frac{\p V^\beta}{\p x_\alpha}|^2+\sum\limits^2_{\beta=1}|\frac{\p V^\beta}{\p x_3}|^2+\sum\limits^2_{\alpha=1}|\frac{\p V^3}{\p x_\alpha}|^2+|\frac{\p V^3}{\p x_3}|^2\big)dx_hdx_3\\
=&\int_{|x_3|\leq |x_h|^\al}\int_{R\leq |x_h|\leq 2R}\Big(\sum\limits^2_{\alpha,\beta=1}|\frac{\p \bar{V}^\beta}{\p \bar{x}_\alpha}|^2({\f{x_h}{R}},\f{x_3}{R^\al})+R^{2-2\al}\sum\limits^2_{\beta=1}|\frac{\p \bar{V}^\beta}{\p \bar{x}^3}|^2({\f{x_h}{R}},\f{x_3}{R^\al})\\
&+\f{1}{R^{2(1-\al)}}\sum\limits^2_{\alpha=1}|\frac{\p \bar{V}^3}{\p \bar{x}^\alpha}|^2({\f{x_h}{R}},\f{x_3}{R^\al})+|\frac{\p \bar{V}^3}{\p \bar{x}_3}|^2({\f{x_h}{R}},\f{x_3}{R^\al})\Big)dx_hdx_3.
\eali
\ees Therefore
\be
\label{e4.6e}
\bali
&\|\na V\|^2_{L^2(\mO_R)}=R^{2+\al}\int_{|\bar{x_3}|\leq |\bar{x}_h|^\al}\int_{1\leq |\bar{x}_h|\leq 2}\Big(\sum\limits^2_{\alpha,\beta=1}|\frac{\p \bar{V}^\beta}{\p \bar{x}^\alpha}|^2(\bar{x}_h, \bar{x}_3)+R^{2(1-\al)}\sum\limits^2_{\beta=1}|\frac{\p \bar{V}^\beta}{\p \bar{x}_3}|^2(\bar{x}_h,\bar{x}_3)\\
&+\f{1}{R^{2(1-\al)}}\sum\limits^2_{\alpha=1}|\frac{\p \bar{V}^3}{\p \bar{x}^\alpha}|^2(\bar{x}_h,\bar{x}_3)+|\frac{\p \bar{V}^3}{\p \bar{x}_3}|^2(\bar{x}_h, \bar{x}_3)\Big)d\bar{x}_hd \bar{x}_3\\
\leq&CR^{4-\al}\|\na \bar{V}\|^2_{L^2(\mO)}.
\eali
\ee
Also it is easy to see
\be\label{e4.7e}
\|f\|^2_{L^2(\mO_R)}=R^{2+\al}\|\bar{f}\|^{2}_{L^2(\mO)}.
\ee
Combining \eqref{e4.6e}, \eqref{e4.7e} and \eqref{e4.3e}, we have
\bes
\|\na V\|^2_{L^2(\mO_R)}\leq CR^{4-\al}\|\na \bar{V}\|^2_{L^2(\mO)}\leq CR^{4-\al}\|\bar{f}\|^2_{L^2(\mO)}=CR^{2(1-\al)}\|f\|^2_{L^2(\mO_R)}.
\ees
This finishes the proof of Proposition \ref{P4.1}. \qed

In Section \ref{subextra}, we will see some applications of Lemma \ref{L4.1L}  and Proposition \ref{P4.1} to flows in a slab (channels of fixed finite depth) or aperture domains.

\subsection{Self-similar solutions}

Another useful class of special solutions to the NS are self-similar solutions and their variants such as discretely self-similar solutions etc. A solution $v$ to the NS is called self-similar if it is invariant under the natural scaling \eqref{nsscal} for all parameters $\lam>0$, namely $v(x, t)=v_\lam(x, t)
\equiv \lam v(\lam x, \lam^2 t), \forall \lam>0, x \in \bR^3$ and all $t>0$ or all $t<0$. In general, if $v=v_\lam$ for one particular $\lam$, then $v$ is called a discretely self-similar solution with factor $\lam$. Given a self-similar $v$, if it is independent of time, it is called  stationary self-similar; otherwise,
if $v$ is defined for $t<0$, then it is called backward self-similar; if $v$ is defined for $t>0$, then it is called forward self-similar.
Stationary self-similar solutions can serve as one model of stationary solutions of NS near a singularity or spatial infinity. A backward one serves as a model of possible type I singularity of solutions of NS;
and a forward self-similar solution can be used to describe long time behavior of solutions of the NS.
Detailed study can be found in \cite{Tsai} Chapter 8.

Self-similar solutions are determined by their profile $U$. In the stationary case
\[
v(x) = |x|^{-1} U(x), \qquad U(x) \equiv v(x/|x|).
\]Here $U$ can be regarded as a vector defined on the unit sphere or a homogeneous vector field on $\bR^3-\{0 \}$ of degree $0$. In the time dependent case
\[
v(x, t) = \lam(t) U(\lam(t) x), \qquad U(x) = v(x, \sgn t), \quad \lam(t)=(t \, \sgn t)^{-1/2}.
\]Let the number $\textbf{m}=0, 1, -1$ for the stationary, forward and backward cases respectively. Then the profiles $U$ satisfy the following stationary equations in $\bR^3$ if $\textbf{m} \neq 0$ and in $\bR^3-\{0\}$ if $\textbf{m}=0$.
\be
\lab{eqselfs}
-\Delta U + {\blue (U \cdot\nabla U)}+\nabla P - \frac{\textbf{m}}{2} U -\frac{\textbf{m}}{2} x  \nabla U =0, \qquad div \,  U=0,
\ee where, as usual $x=(x_1, x_2, x_3)$, $U=(U^1, U^2, U^3)$ and $x \nabla U = \sum^3_{i=1} x_i  \p_{i} U$.

There is a family of stationary self-similar solutions, called Landau solutions \cite{Lan44}, which, in the spherical coordinates,  are explicitly given by
\be
\lab{lansol}
\bali
U^b&= curl  ( L \, e_\theta)=\frac{1}{\rho \sin \phi} \p_{\phi} (L \, \sin \phi) e_\rho -
\frac{1}{\rho} \p_{\rho} (\rho L) e_{\phi}, \\
 &\quad L =  \frac{2 \sin \phi}{(a - \cos \phi)}, \quad b=16 \pi \left(a+\frac{a^2}{2} \ln \frac{a-1}{a+1} + \frac{4a}{3(a^2-1)}\right) e_3, \quad a>1.
\eali
\ee With the parameter $b$, $U^b$ actually solves the nonhomogeneous problem in $\bR^3$:
\[
-\Delta U^b + {\blue ( U^b\cdot \nabla)} U^b+\nabla P^b  =b \delta, \qquad div \,  U^b=0,
\]where $\delta=\delta(x)$ is the Dirac delta function centered at $0$.
 The spherical system we take is the standard one
\[
x=(x_1, x_2, x_3)= (\rho \sin \phi \cos \th, \rho \sin \phi \sin \th, \rho \cos \phi).
\]Here $\rho=|x|$, $\phi$ is the polar angle between $x$ and $e_3$, $\th$ is the azimuthal angle.
This system is a convenient one for studying self-similar solutions.

The Landau solutions are axially symmetric without swirl.  The following result due to Sverak \cite{Sv11}
shows that they are the only (-1)-homogenous ones. Earlier Tian and Xin proved in \cite{TX}  that all (-1)-homogeneous, axially symmetric  solutions  in $C^2(\bR^3 \backslash \{0\})$ are Landau solutions.
Here we are using the notion of "(-$\textbf{m}$)-homogeneous" interchangeably with "homogeneous of degree
$-\textbf{m}$".

\begin{theorem}
If a homogeneous vector field $v$ of degree $-1$ is a solution to the stationary NS in $\bR^3 \backslash 0$,
then $v=U^b$ for some $b \in \bR^3$.
\end{theorem}

Extension of the theorem to (-1)-homogenous solutions with singularities at the poles of ${\mathbb S}^2$
can be found in recent papers by Li, Li and Yan \cite{LLY18, LLY18-2}.

Backward self-similar solutions of the form
\be
\lab{le-bss}
v(x, t) = \frac{1}{\sqrt{T-t}} U\left(\frac{x}{\sqrt{T-t}}\right), \quad t \in [t_0, T),  x \in \bR^3,
\ee was proposed by Leray as a candidate for possible type I singularity of the time dependent NS at time $T$. Here $t_0$ is the initial time. Leray asked whether there is a self-similar solution of the NS in above form with finite energy, i.e.
\eqref{enineq} holds. For such solutions $U \in L^2(\bR^3)$. In 1996, Necas, Ruzicka and Sverak \cite{NRS}  proved that the only such solution is $0$. In \cite{Tsai98}, T. P. Tsai generalized the result to very weak self-similar solutions satisfying the local finite energy condition. For a definition of very weak solutions, see \cite{Tsai} p53 e.g.

\begin{theorem}
Suppose $v$ is a self-similar very weak solution of the NS with zero force in the cylinder $B_1  \times (-1, 0) \subset \bR^3 \times (-1, 0)$. It is zero if it has finite local energy
\[
\sup_{-1<t<0} \int_{B_1} |v(x, t)|^2 dx + 2 \int^0_{-1} \int |\nabla v(x, t)|^2 dxdt  < \infty,
\]or if the profile $U \in L^q(\bR^3)$ for some $q \in [3, \infty)$. Also $U$ is constant if $U \in L^\infty(\bR^3)$.
\end{theorem}

The proof is based on the fact that the total head pressure
\[
Q= \frac{1}{2} | U(x)|^2 + P(x) + \frac{1}{2} x \cdot U(x)
\]satisfies the equation
\[
-\Delta Q + (U(x) +\frac{1}{2} x) \cdot \nabla Q(x) = - | curl \, U(x)|^2 \le 0.
\]If one can prove that $Q$ behaves well at infinity, then the maximum principle and finite energy assumption will force $U=0$.

In the forward self-similar case, Jia and Sverak proved the first  existence result
for large self-similar initial values in $C^\alpha(\bR^3\backslash \{0\})$.

\begin{theorem}[\cite{JiSv}]
Assume $v_0$ is scale-invariant and locally H\"older continuous in $\bR^3\backslash \{0\}$  with $div \, v_0= 0$ in $\bR^3$. Then the Cauchy problem \eqref{nse} has at least one scale-invariant solution $v$ which is smooth in $\bR^3 \times (0,\infty)$ and locally H\"older continuous in $(\bR^3 \times ([0,\infty))\backslash \{(0,0)\}$.
\end{theorem}

The proof uses Leray-Schrauder fixed point theorem to solve equation \eqref{eqselfs}
with $\textbf{m}=1$ in a weighted neighborhood around $\mu e^{\Delta} v_0= \mu \int_{\bR^3} G(x, 1, y) v_0(y) dy$. Here $\mu \in (0, 1]$ is a parameter and
$G$ is the standard heat kernel on $\bR^3$.  More specifically, one looks for $U_\mu$ solving \eqref{eqselfs} such that
\[
\left|\nabla^\alpha \left( U_\mu(x)- \mu e^{\Delta} v_0(x) \right)\right| \le \frac{C(\alpha, v_0)}{(1 +|x|)^{2+\alpha}}, \qquad
\alpha=0, 1.
\]One can convert this to a fixed point problem of an integral operator involving the Stokes kernel, which is easy to solve if $\mu$ is small. If one has compactness for the integral operator, then Leray Schauder theorem ensures the solution exists for $\mu=1$. The compactness is based on a priori H\"older estimates
near the initial time for the local Leray solutions from the book \cite{Leri1}. Then $v= t^{-1/2}U_1( t^{-1/2} x)$ is a solution in the above theorem.

These solutions potentially have application in searching non-unique Leray Hopf solutions of NS. Denote by $v_\mu=t^{-1/2}U_\mu( t^{-1/2} x)$ the unique scale-invariant solution with initial data $\mu v_0$ and $\mu$ sufficiently close to $0$. Note for large $\mu$ the uniqueness is unknown.
In the paper \cite{JiSv15}, Jia and Sverak consider solutions of NS in the form $v_\mu(x,t)+t^{-1/2} \phi(t^{-1/2} x,t)$. The linearization of the equation for $\phi$ takes the form $t \partial_t \phi
=L_\mu \phi$. For small enough $\mu$ the eigenvalues of $L_\mu$ are away from the imaginary axis with real part less than zero. They suggest two potential scenarios under which the solution curve can be continued as a regular function of $\mu$ and the eigenvalues of $L_\mu$ can cross the imaginary axis. Under these hypotheses two solutions of the NS are obtained with initial data $\mu v_0$; Although these have infinite energy but by analysis on certain critical singularities in some lower order terms, the authors are then able to localize these solutions to obtain possible non-uniqueness for Leray-Hopf solutions.

\subsection{Addendum: Extra results on D solutions, by Xinghong Pan and Q. S. Zhang}
\lab{subextra}
$$
$$

In this subsection, we present some new vanishing results on D solutions which extend those in \cite{CPZ2018-1}. The improvements are on wider domains and relaxation of symmetry or growth condition on the solutions. These results seem to be new and are not presented elsewhere in the literature. The general idea of the proof is similar to that of the previous results.

 For Theorem \ref{thzp}, we can weaken the axially symmetric assumption on the three components of $v^r,v^\th, v^3$ to the case that only $v^\th$ is axially symmetric.
  Denote $x_h=(x_1,x_2)$ and $\bS$ the 1 dimensional periodic domain with period being $1$. We have the following result.

%let $\overline{\O}$ be a two dimensional smooth exterior domain, whose complementary set contains a neighbourhood of the origin.

\begin{theorem}[vanishing of periodic HDS with just axial symmetry of $v^\th$]
\label{thutheta}
Let $(v, P)$ be a bounded smooth solution to the problem
\begin{equation}\label{NSP1}
\left\{
\begin{array}{ll}
{\blue (v\cdot\nabla)} v+\nabla P-\Delta v=0,\quad \nabla\cdot v=0, &\text{ in }\quad \bR^2\times\bS, \\
v(x_h,x_3)=v(x_h,x_3+1), &\text{ in}\q\ \bR^2\times\bS,
\end{array}
\right.
\end{equation}
with finite Dirichlet integral
\be\label{D-COND}
\int_{\bR^2\times\bS}|\nabla v(x)|^2dx<+\i.
\ee
If just $v^{\th}$ is axially symmetric (independent of $\th$), then we have $v \equiv c e_3$.
\end{theorem}

%\begin{remark}
%We emphasize here that our theorem \ref{thm1.1} can also be applied to the case $\oO=\bR^2$ with $\eqref{NSP1}_4$ replaced by
%\bes
%\lim\limits_{|x_h|\rightarrow+\i} u=0.
%\ees
%\end{remark}
The next theorem deals with ASHDS in an aperture domain $\O:=\{x | x_h\in\bR^2,  |x_3|\leq \lt(\max(1,r)\rt)^\al\}$ for some $\al \ge 0$. Study of flows in aperture domains is also of usefulness, as mentioned in \cite{Galdi2011} Chapter XIII.

\begin{theorem}[ASHDS on aperture domains]
\label{thmexpanding}
Let $(v, P)$ be an axially symmetric bounded smooth solution to the problem
\begin{equation}
\label{NSD}
{\blue (v\cdot\nabla)} v+\nabla P-\Delta v=0, \quad
\nabla\cdot v=0 \quad \text{in } \, \O; \quad
v(x)=0  \quad \text{on}\q \p\O
\end{equation}
with finite Dirichlet integral \eqref{edcondition}.
Then we have
(i) $v \equiv 0$ if $0\leq \al<1/2$;
(ii) $v^\th \equiv 0$ if $0\leq \al< 3/4$.
\end{theorem}

When $\al=0$, we can remove the axial symmetry assumption of the solution, which goes back to Theorem \ref{ths2}.

\subsubsection{Proof of Theorem \ref{thutheta}, one component axially symmetric periodic solutions}\ \\

The proof is divided into two steps.

\noindent {\it Step 1.  ${L^2}$ boundedness of ${v^r}$ and $L^2$ mean oscillation of $P$.}

\begin{lemma}\label{Lemma2.1}
For $n\in\bN/\{0\}$, denote $\mC_n=\{x_h | \, n\leq |x_h|\leq 2n\}$. Let $(v,P)$ be the solution of \eqref{NSP1}. Under the assumptions of Theorem \ref{thutheta}, we have
\be\label{evrL2}
 \|v^r\|_{L^2({\bR^2\times\bS})}<C_\ast,
\ee
and
\be\label{epl2}
\|P-P_n\|_{L^2(\mC_n\times\bS)}\leq C_\ast n,
\ee
where $C_\ast=C(\|v\|_{L^\i},\|\na v\|_{L^2})$ and $P_n:=\f{1}{|\mC_n\times\bS|}\int_{\mC_n\times\bS}Pdx$ is the average of $P$ on $\mC_n\times \bS$.

\end{lemma}
\begin{proof}
In cylindrical coordinates, if $v^\th$ is independent of $\th$, the divergence free condition $\eqref{NSP1}_2$ is translated as
\bes
\nabla\cdot v=\f{1}{r}\p_r(rv^r)+\frac{1}{r}\p_\th v^\th+\p_{x_3}v^3=\f{1}{r}\p_r(rv^r)+\p_{x_3}v^3=0.
\ees
Integrating the above inequality in $\bS$ about $x_3$ and by using the periodic boundary condition, we can deduce
$
\int_{\bS} v^r dx_3=0.
$

Then the one dimensional Poincar\'e inequality indicates that
\bes
\bali
\int_{\bR^2\times\bS} |v^r|^2dx=\int_{\bR^2\times \bS} \lt|v^r-\f{1}{|\bS|}\int v^rdx_3\rt|^2dx_3dx_h
\ls \int_{\bR^2}\int^1_0 |\p_{x_3}v^r|^2dx_3dx_h<\infty.
\eali
\ees
This proves \eqref{evrL2}.

In Proposition \ref{P4.1}, if $\al=1$ and the domain $\mO_R$ is replaced by $\O_R=\big\{x| R\leq |x_h|\leq 2R,\ |x_3|\leq 2R\big\}$, it is not hard to see the proof is still valid. Now set $R=n$ with $n\in\bN\backslash\{0\}$. From Proposition \ref{P4.1}, there exists a $V \in H^1_0(\O_n)$ satisfying $\nabla \cdot V =f$  and \eqref{e4.4e} with $f=P-P_n$ and $\mO_R$ replaced by $\O_n$.

 Now multiplying $\eqref{NSP1}_1$ with $V$ and integration on $\O_n$, we get
\bes
\bali
\int_{\O_n}\na(P-P_n)\cdot V dx=&\int_{\O_n}(\Dl v-v\cdot\na v)\cdot V dx.
                               \eali
\ees
Integration by parts and \eqref{e4.4e} indicate that
\be\label{e2.9}
\bali
\int_{\O_n}(P-P_n)^2dx=&\int_{\O_n}(P-P_n)\na\cdot V dx=-\int_{\O_n}\big(\Dl v-v\cdot\na v\big)\cdot V dx\\
=&\int_{\O_n}\sum^3_{i,j=1}\p_iv^j\p_iV^j+\p_i(v^i v^j) V^j dx= \int_{\O_n}\sum^3_{i,j=1}\big(\p_iv^j-v^iv^j\big)\p_iV^j dx\\
%\leq&\|\na V\|_{L^2(\O_n)}\big(\|\na u\|_{L^2(\O_n)}+\||u|^2\|_{L^2(\O_n)}\big)\\
\leq&\|\na V\|_{L^2(\O_n)}\big(\|\na v\|_{L^2(\O_n)}+\|v\|^2_{L^\i}\|1\|_{L^2(\O_n)}\big)\\
\leq& C \|P-P_n\|_{L^2(\O_n)}\big(\|\na v\|_{L^2(\O_n)}+\|v\|^2_{L^\i}\|1\|_{L^2(\O_n)}\big)
\eali
\ee
%Using Cauchy-schwartz inequality, we have
%\be
%\bali
%&\int_{\O_n}(p-p_n)^2dx\\
%\leq&\|\na V\|_{L^2(\O_n)}\big(\|\na u\|_{L^2(\O_n)}+n^{3/2}\|u\|^2_{L^\i(\O_n)}\big)\\
%\leq& \ve\|\na V\|^2_{L^2(\O_n)}+C_\ve \big(\|\na u\|^2_{L^2(\O_n)}+\|u\|^4_{L^\i(\O_n)}n^3\big)\\
%\leq& C\ve\|p-p_n\|^2_{L^2(\O_n)}+C_\ve \big(\|\na u\|^2_{L^2(\O_n)}+\|u\|^4_{L^\i(\O_n)}n^3\big).
%\eali
%\ee
Then we can obtain
 \bes
\|P-P_n\|_{L^2(\O_n)}\leq C \big(\|\na v\|_{L^2(\O_n)}+\|v\|^2_{L^\i(\O_n)}n^{3/2}\big).
 \ees
Remembering that $(v,P)$ is periodic in the $x_3$ direction, the above inequality can be rewritten as
\bes
n^{1/2}\|P-P_n\|_{L^2(\mC_n\times\bS)}\leq C\big(n^{1/2}\|\na v\|_{L^2(\mC_n\times\bS)}+\|v\|^2_{L^\i(\mC_n\times\bS)}n^{3/2}\big).
 \ees
This implies \eqref{epl2}.
\end{proof}

\medskip

\noindent{\it Step 2. Trivialness of $v$}

 Let $\phi(s)$ be a smooth cut-off function satisfying
\be\label{testf}
\phi(s)=
1\q s\in [0,1];  \quad
 \phi(s) = 0 \q   s\geq 2,
\ee with the usual property that $\phi$, $ \phi'$ and $\phi''$ are bounded.
Set $\phi_n(y_h)=\phi(\f{|y_h|}{n})$. Testing the NS in $\eqref{NSP1}$
with $v \phi_n$, we find that
\bes
\begin{split}
&\int_{\2O\times\bS}|\nabla v|^2\phi_n dx
=-\sum\limits^3_{i=1}\int_{\2O\times\bS}v^i\na v^i\cdot\na\phi_ndx+\int_{\2O\times\bS}\left(\frac{1}{2}|v|^2+(P-P_{n})\right) v \cdot\nabla\phi_ndx.
\end{split}
\ees Here $\2O$ is the 2 dimensional ball $\{ x_h \, | \, |x_h|<2n \}$.
It follows that
\be\label{e4.4}
\begin{split}
&\int_{\2O\times\bS}|\nabla v|^2\phi_ndx
\ls \int_{\mC_{n}\times\bS}|v||\na v||\na\phi_n|dx
+\int_{\mC_{n}\times\bS}|v \cdot\na \phi_n| \, |v|^2dx+\int_{\mC_{n}\times\bS}|P-P_{n}| \, |v \cdot\na\phi_n|dx\\
:=&I_1+I_2+I_3.
\end{split}
\ee
Observe that, as $n \to \infty$,
\[\label{e4.6}
\begin{split}
I_1 \lesssim\f{\|v\|_{L^\i(\mC_{n}\times\bS)}}{n}\|\na v \|_{L^2(\mC_{n}\times\bS)}\|1\|_{L^2((\mC_{n}\times\bS)}
\ls\|v\|_{L^\i(\mC_{n}\times\bS)}\|\na v \|_{L^2(\mC_{n}\times\bS)}\to 0;
\end{split}
\]

\[\label{e4.7}
\begin{split}
I_2\lesssim \| v \|^2_{L^\i(\mC_{n}\times\bS)}\int_{\mC_{n}\times\bS}|v^r\p_r \phi_n|dx
\ls\f{\|v\|^2_{L^\i(\mC_{n}\times\bS)}}{n}\| v^r\|_{L^2(\mC_{n}\times\bS)}\|1\|_{L^2(\mC_{n}\times\bS)}
\ls\| v \|^2_{L^\i(\mC_{n}\times\bS)}\| v^r\|_{L^2(\mC_{n}\times\bS)}\to 0;
\end{split}
\]

\[\label{e4.8}
\begin{split}
I_3\lesssim \int_{\mC_{n}\times\bS}|P-P_{n}| \, |v^r\p_r\phi_n|dx
\ls\f{1}{n}\|P-P_{n}\|_{L^2(\mC_{n}\times\bS)}\|v^r\|_{L^2(\mC_{n}\times\bS)}
\ls C_0\| v^r\|_{L^2(\mC_{n}\times\bS)}\to 0.
\end{split}
\]

Here we have applied Cauchy-Schwarz inequality, the boundedness of the oscillation of $P$ in dyadic annulus from Lemma \ref{Lemma2.1} and $v^r\in L^2(\2O\times\bS)$.
Combining those estimates of $I_1$, $I_2$ and $I_3$, \eqref{e4.4} yields
$
\int_{\2O\times\bS}|\nabla v|^2dx=0,
$
 which means $v$ is a constant vector. The $v^r=0$ by Lemma \ref{Lemma2.1} and $v^{\th}=0$ by axial symmetry. So actually we have
$
v \equiv c e_3.
$ \qed

\subsubsection{Proof of Theorem \ref{thmexpanding},  axially symmetric solutions on aperture domains}
$ $

\medskip

First we show an a priori decay of ${v}$.

\begin{proposition}\label{propdu}
Under the assumption of Theorem \ref{thmexpanding}, we have
\be\label{du}
v= O\Big(\f{\ln^{1/2} r}{r^{1/2}}\Big)\q \text{as}\ r\rightarrow +\i.
\ee
\proof
\end{proposition}
The proof is based on $Brezis-Gallouet$ inequality \cite{BG:1} and its refinement, together with scaling and dimension reduction techniques, which is similar to \cite[part 5.2, p. 1403-1406]{CPZ2018-1}.

Fixing $x_0 \in \O:=\{x||x_h|\in\2O,|x_3|\leq \max \{1, r^\al\}$ such that $|x'_0|=  r_0$ is large. Without loss of generality, we can assume, in the cylindrical coordinates, that $x_0=(r_0, 0, 0)$, i.e. $(x_3)_0=0$, $\theta_0=0$.
Consider the scaled  solution
$
\t {v} ( \t{x}) = r_0 v( r_0 \t{x})
$ which is also axially symmetric. Hence $\t {v}$ can be regarded as a two variable function  of the scaled variables $\t{r}, \t{x_3}$. Consider the two dimensional domain
\bes
\bar{\mathcal{D}}_0=\big\{(\t{r},\t{x_3}): 1-\f{1}{2}r_0^{\al-1}\leq \t{r}\leq1+ \f{1}{2}r_0^{\al-1},\ |\t{x_3}|\leq r_0^{\al-1}\t{r}^{\al}\big\}.
\ees

Then for $\t{v}=\t{v}(\t{r}, \t{x_3})$, we have $\t{v}(1, 0)=r_0 v(x_0)$.
Recall that $\t{v}$ satisfies the Dirichlet boundary condition.
Applying the refined Brezis-Gallouet inequality (Lemma \ref{leRBG})
 on $\bar{\mathcal{D}}_0$, after a simple adjustment on constants, we can find an absolute constant $C$ such that
\[
\bali
| \t{v}(1, 0) | \leq C \left[ \left( \int_{\bar{\mathcal{D}}_0} |\t{\nabla} \t{v} |^2 d\t{r}d\t{x_3} \right)^{1/2}
 +1  \right] \times
\log^{1/2}\left[ r_0^{\al-1}\left( \int_{\bar{\mathcal{D}}_0} |\t{\Delta} \t{v} |^2 d\t{r}d\t{x_3} \right)^{1/2} + e  \right],
\eali
\]where $\t \nabla = (\partial_{\t{r}}, \partial_{\t{x_3}})$ and
$\t \Delta = \partial^2_{\t{r}} + \partial^2_{\t{x_3}}$. From this and the assumption that
$1/2 \le \t{r} \le 2$, we see that
\[
\bali
| \t{v}(1, 0) | \le C \left[ \left( \int_{\bar{\mathcal{D}}_0} |\t{\nabla} \t{v} |^2 \t{r} d\t{r}d\t{x_3} \right)^{1/2}  +1  \right] \times
 \log^{1/2}\left[ r_0^{\al-1}\left( \int_{\bar{\mathcal{D}}_0} |\t{\Delta} \t{v} |^2 \t{r} d\t{r}d\t{x_3} \right)^{1/2} + e  \right].
\eali
\]Now we can scale this inequality back to the original solution $u$ and variables $r=r_0 \t{r}$ and $z=r_0 \t{x_3}$ to get
\[
\bali
r_0 |v(x_0)| \le C  \left[ \sqrt{r_0} \left( \int_{D_0} |\nabla v |^2 r drdx_3 \right)^{1/2}  +1  \right] \times
 \log^{1/2}\left[ r^{\al+1/2}_0 \left( \int_{D_0} (|\p^2_r v|^2 + |\p^2_{x_3} v|^2) r drdx_3 \right)^{1/2}+ e  \right],
\eali
\]where
$
D_0=\big\{({r},{x_3}): r_0-\f{1}{2}r_0^ {\al}\leq {r}\leq r_0+ \f{1}{2}r_0^{\al},\ |{x_3}|\leq {r}^{\al}\big\}.
$
By condition \eqref{edcondition}, this proves the claimed decay of velocity. Note that by our assumption in the theorem, the solution $u$ is globally bounded and then it is not hard to prove that the first and second derivative of $v$ are also bounded. \qed

The rest of the proof of Theorem \ref{thmexpanding} is divided into two subsections.

\subsubsection{\bf Proof of case (i): $0\leq \al <1/2$.}

First we will give a $L^2$ estimate of $v$ and mean oscillation of the pressure $P$ by using the preceding Proposition \ref{P4.1}.

\begin{lemma}
\label{leupl3}
Let $(v, P)$ be the solution of \eqref{NSD} and  $\mO_R=\big\{x \, | \, R\leq |x_h|\leq 2R,\ |x_3|\leq |x_h|^\al\big\}$, then we have
\be
\label{ul2}
\int_{\mO_R} |v|^2dx\leq  o( R^{2\alpha}),
\ee

\be\label{e4.9e}
\|P-P_R\|_{L^2(\mO_R)}\leq o(1) R^{(1-\al)}(1+R^{\al-1/2}\ln^{1/2} R),
\ee
as $R\rightarrow\i$. Here $P_R:=\f{1}{|\mO_R|}\int_{\mO_R} P dx$ is the average of $P$ on $\mO_R$.
\end{lemma}
\pf
Since we have zero boundary, the one dimensional Poincar\'e inequality indicates that
\bes
\bali
\int_{\mO_R} |v|^2dx&=\int_{R\leq |x_h|\leq 2R}\int_{|x_3|\leq r^\al} |v|^2dx_3dx_h
\ls R^{2\al} \int_{R\leq |x_h|\leq 2R}\int_{|x_3|\leq r^\al}|\p_{x_3} v|^2dx_3dx_h
\ls o(R^{2\al}).
\eali
\ees
 by the definition of D solutions. This proves (\ref{ul2}).

From Proposition \ref{P4.1}, there exists a $V \in H^1_0(\mO_R)$ satisfying $ \nabla \cdot V =f$ and \eqref{e4.4e} with $f=P-P_R$. Now multiplying $\eqref{NSD}_1$ with $V$ and integrating on $\mO_R$, we get
\bes
\int_{\mO_R}\na(P-P_R)\cdot V dx=\int_{\mO_R}(\Dl v-v\cdot\na v)\cdot V dx.
\ees
The same derivation as \eqref{e2.9} indicates that
\bes
\bali
\int_{\mO_R}(P-P_R)^2dx=&\int_{\mO_R}(P-P_R)\na\cdot V dx\\
%=&\int_{\mO_R}\sum^3_{i,j=1}\p_iu^j\p_iV^j+\na\cdot(u\otimes u)\cdot V dx\\
%=& \int_{\mO_R}\sum^3_{i,j=1}(\p_iu^j-u^iu^j)\p_iV^j dx\\
\leq&\|\na V\|_{L^2(\mO_R)}\big(\|\na v\|_{L^2(\mO_R)}+\|v\|_{L^\i(\mO_R)}\| v\|_{L^2(\mO_R)}\big)\\
\leq& CR^{1-\al}\|P-P_R\|_{L^2(\mO_R)}\big(\|\na v\|_{L^2(\mO_R)}+\|v\|_{L^\i(\mO_R)}\| v\|_{L^2(\mO_R)}\big).
%\leq& \f{\ve}{R^{2(1-\al)}}\|\na V\|^2_{L^2(\mO_R)}+C_\ve R^{2(1-\al)}\big(\|\na u\|_{L^2(\mO_R)}+\|u\|_{L^\i(\mO_R)}\| u\|_{L^2(\mO_R)}\big)^2\\
%\leq& C\ve\|p-p_R\|^2_{L^2(\mO_R)}+C_\ve R^{2(1-\al)}\big(\|\na u\|_{L^2(\mO_R)}+\|u\|_{L^\i(\mO_R)}\| u\|_{L^2(\mO_R)}\big)^2.
\eali
\ees
 Here to reach the last line, we used \eqref{e4.4e}. Then we can obtain

\bes
\bali
\|{\blue P-P_R}\|_{L^2(\mO_R)}\leq& C R^{(1-\al)}\big(\|\na v\|_{L^2(\O_R)}+\|v\|_{L^\i(\O_R)}\| v\|_{L^2(\O_R)}\big)\\
= & o(1) R^{(1-\al)}(1+R^{\al-1/2}\ln^{1/2} R),
\eali
\ees
which is \eqref{e4.9e}. \qed

\noindent\textbf{Vanishing of ${v}$}

Now we are in a position to complete the proof of case (i) of Theorem \ref{thmexpanding}. Let $\phi(s)$ be the smooth cut-off function defined in \eqref{testf}.
Set $\phi_R(y')=\phi(\f{|y'|}{R})$ where $R$ is a large positive number.
 Now testing the NS in \eqref{NSD}
with $v\phi_R$, we obtain
\bes
\int_{\O}-\Dl v (v\phi_R)dx=\int_{\O}-\lt( {\blue (v\cdot\nabla)v}+\na (P-P_R)\rt ) (v\phi_R)dx.
\ees
Integration by parts yields that
\bes
\begin{aligned}
&\q\int_{\O}|\na v|^2 \phi_R dx-\f{1}{2}\int_{\O}|v|^2\Dl\phi_R dx
=\f{1}{2}\int_{\O} |v|^2v\cdot\na\phi_R dx+\int_{\O} (P-P_R)v\cdot\na \phi_R dx.
\end{aligned}
\ees
 Then we have, since
$\phi_R$ depends only on $r$, that
\bes
\begin{aligned}
&\q\int_{\O}|\na v|^2 \phi_Rdx\\
&\ls \f{1}{R^2}\int_{\mO_R}|v|^2dx+
\f{1}{R}\int_{\mO_R} |v^r| \, |v|^2dx+ \f{1}{R}\int_{\mO_R}|P-P_R|\, |v^r|dx\\
&\ls \f{1}{R^2} \int_{\mO_R} |v|^2 dx +
\f{\Vert v^r \Vert_\infty}{R}
\int_{\mO_R} |v|^2 dx+\f{1}{R} \left(\int_{\mO_R} (v^r)^2 dx \right)^{1/2}
\left(\int_{\mO_R}|P-P_R|^2dx \right)^{1/2}\\
&\ls\big(R^{2\al-2}+R^{2\al-3/2}\ln^{1/2} R+ o(1)+R^{\al-1/2}\ln^{1/2} R\big).
\end{aligned}
\ees
When $0\leq\al<1/2$, let $R\rightarrow +\i$, we arrive at
$
\int_{\O}|\na v|^2dx=0,
$
which shows that $v\equiv c.$ Besides, recall $v=0$ at the boundary, then at last we deduce $v\equiv 0.$ This completes the proof of case (i) of Theorem \ref{thmexpanding}.
\medskip

\subsubsection{\bf Proof of case (ii): $1/2\leq \al <3/4$.}

The proof is divided into 2 steps.

\noindent{\it Step 1. Caccioppoli inequality for ${\G=r v^\th}$} \label{ssuth1}

First for $R>2$, pick the domains
$
D(R)=\O\cap (\underline{B}_R\times\bR),
$ and
$
D(R_1,R_2)=D(R_1)\backslash D(R_2),\ \ R_1>R_2>2.
$ Here $\underline{B}_R$ is the 2 dimensional ball centered at the origin in $\bR^2$, with radius $R$.
We first derive a Caccioppoli type energy estimates of $\G$ by equation \eqref{eqvth}. Consider a standard test function $\psi(r)$ satisfying
$
\operatorname{supp} \psi \subset \underline{B}_{\sigma_1}$, $\psi=1 \ \textrm{in}\  \underline{B}_{\sigma_2}, 0\leq \psi \leq 1; 0<\sigma_2<\sigma_1 \le 1;
$
$
|\nabla^k\psi| \leq  \frac{C}{(\sigma_1-\sigma_2)^k}\ \  \text{for}\ \ k=1,2.              \\
$
\begin{lemma}
Let $\frac{1}{2}\leq \sigma_2<\sigma_1\leq 1$, $R\geq 2 $ and $\psi_R(x)=\psi(\f{r}{R})$. Denote $f:=|\G|^q$ for $q>1$. Then we have
\bes
\int_{D(\sigma_1R)}|\nabla (f\psi_R)|^2\lesssim\f{(1+\|v^r\|_{L^\i} R^\alpha)^2 }{(\sigma_1-\sigma_2)^2R^2} \int_{D(\sigma_1R,\sigma_2R)}f^2,
\ees here and in the proof $\|v^r\|_{L^\i}=\|v^r\|_{L^\i(D(\sigma_1R,\sigma_2R))}$.
\end{lemma}
\begin{proof} Recall \eqref{eqvth} without time variable is:
\[
\Delta \Gamma - b \nabla \Gamma- \frac{2}{r} \p_r
\Gamma=0.
\]with $b=v^r e_r + v^3 e_3$.
Testing this by $q|\G|^{2q-2}\G\psi^2_R$ in $\O$ gives
\begin{equation}\label{e5.2}
\begin{aligned}
\q\frac{1}{2}\int_{\O} (b\cdot\nabla f^2+\frac{2}{r}\p_r f^2)\psi^2_Rdx
=q\int_{\O}\Delta \G |\G|^{2q-2}\G\psi^2_Rdx.
\end{aligned}
\end{equation}
Later the integral variables $dx$ will not be written out for simplicity unless there is confusion. Using Cauchy-Schwartz's inequality and integration by parts, we have
\be\label{e5.3}
\bali
&q\int_{\O}\Delta \G |\G|^{2q-2}\G\psi^2_R=q\int_{D(\sigma_1R)}\Delta |\G| |\G|^{2q-1}\psi^2_R     \\
=&-q\int_{D(\sigma_1R)} (2q-1)|\nabla |\G| |^2\G^{2q-2}\psi^2_R+\nabla|\G||\G|^{2q-1}\cdot\nabla\psi^2_R       \\
=&-\int_{D(\sigma_1R)}(2-\frac{1}{q})|\nabla(f\psi_R)|^2-(2-\frac{2}{q})f\nabla \psi_R\cdot\nabla(f\psi_R)-\frac{1}{q}f^2|\nabla\psi_R|^2   \\
\lesssim& -\int_{D(\sigma_1R)}|\nabla(f\psi_R)|^2+C \int_{D(\sigma_1R)} f^2|\nabla \psi_R|^2.
\eali
\ee
Also we have
\be\label{e5.5}
\bali
-\int_{\O}&\frac{1}{r}\p_r f^2\psi^2_R=-2\pi\int^{\sigma_1R}_{0}\int_{|x_3|\leq (\max\{1,r\})^\al}\p_r f^2\psi^2_R dx_3 dr\\
\leq & 2\pi\int^{\sigma_1R}_{0}\int_{|x_3|\leq (\max\{1,r\})^\al} f^2\p_r\psi^2_R dx_3 dr
\lesssim \int_{D(\sigma_1R,\sigma_2R)} f^2(|\nabla\psi_R|^2)
\eali
\ee
Consequently, using \eqref{e5.2} and combining \eqref{e5.3} and \eqref{e5.5}, we get
\be\label{e5.5e}
\bali
&\int_{D(\sigma_1R)}|\nabla(f\psi_R)|^2
\leq C\int_{D(\sigma_1 R,\sigma_2 R)} f^2 \, |\nabla\psi_R|^2 -\frac{1}{2}\int_{D(\sigma_1R)} b\cdot\nabla f^2\psi^2_R  \\
\leq&\frac{C}{(\sigma_1-\sigma_2)^2R^2}\int_{D(\sigma_1R,\sigma_2R)} f^2-\frac{1}{2}\int_{D(\sigma_1R)} b\cdot\nabla f^2\psi^2_R.
\eali
\ee
By the divergence-free property of the drift term $b$, we have
\be\label{e5.6e}
\bali
&-\frac{1}{2}\int_{\O} b\cdot\nabla f^2\psi^2_R
=\int_{D(\sigma_1R)} f^2\psi_Rb\cdot\nabla\psi_R=\int_{D(\sigma_1R,\sigma_2R)} (f\psi_R)f v^r\p_r\psi_R     \\
\leq&C\f{\|v^r\|_{L^\i}}{(\sigma_1-\sigma_2)R} \|f\psi_R\|_{L^2(D(\sigma_1R,\sigma_2R))}\|f\|_{L^2(D(\sigma_1R, \sigma_2R))}\\
\leq& C_\ve\f{\|v^r\|^2_{L^\i}R^{2\al}}{(\sigma_1-\sigma_2)^2R^2} \|f\|^2_{L^2(D(\sigma_1R, \sigma_2R))}+\ve R^{-2\al}\|f\psi_R\|^2_{L^2(D(\sigma_1 R, \sigma_2 R))}\\
\leq& C_\ve\f{\|v^r\|^2_{L^\i}R^{2\al}}{(\sigma_1-\sigma_2)^2R^2} \|f\|^2_{L^2(D(\sigma_1R,\sigma_2R))}+C\ve\|\na(f\psi_R)\|^2_{L^2(D(\sigma_1 R, \sigma_2 R))};
\eali
\ee
where at the third line of \eqref{e5.6e}, we have used the fact $f=0$ on the boundary of $\O$, and the following 1-dimensional Poincar\'e inequality.
\be\label{epoincare}
\bali
&\q\|f\psi_R\|^2_{L^2(D(\sigma_1 R, \sigma_2 R))}
=\int^{2\pi}_0 \int^{\sigma_1R}_{\sigma_2 R}\int^{r^\al}_{-r^\al}|f\psi_R|^2 dx_3 r drd\th\\
&\leq CR^{2\al}\int^{2\pi}_0\int^{\sigma_1R}_{0}\int^{r^\al}_{-r^\al}|\p_{x_3}(f\psi_R)|^2 dx_3 rdrd\th\leq CR^{2\al}\|\na(f\psi_R)\|^2_{L^2(D(\sigma_1R))}.
\eali
\ee
Combining \eqref{e5.5e} and \eqref{e5.6e}, by choosing small $\ve$, we get
\be\label{e5.9}
\int_{D(\sigma_1R)}|\nabla(f\psi_R)|^2\lesssim\f{(1+\|v^r\|_{L^\i}R^\al)^2}{(\sigma_1-\sigma_2)^2R^2} \int_{D(\sigma_1R,\sigma_2R)}f^2.
\ee
\end{proof}

\noindent{\it  Step 2. Vanishing of $v^\th$}

We will show that $\Gamma$ decays to $0$ as $r \to \infty$. Then the maximum principle implies $\Gamma$ and hence $v^\th=0$.
We use Moser's iteration.

 By H\"{o}lder inequality and Sobolev imbedding theorem, for any $0<\beta<4$, we have
\be\label{ebeta}
\bali
&\|f\psi_R\|^{2+\beta}_{L^{2+\beta}(D(\sigma_1R))}=\int_{D(\sigma_1R)}|f\psi_R|^{2+\beta}=\int_{D(\sigma_1R)}|f\psi_R|^{2-\f{\beta}{2}}|f\psi_R|^{\f{3}{2}\beta} \\
\ls&\big(\int_{D(\sigma_1R)} |f\psi_R|^{2}dy\big)^{1-\f{\beta}{4}}\big(\int_{D(\sigma_1R)}|f\psi_R|^{6}dy\big)^{\beta/4}
\ls\|f\psi_R\|^{2-\f{\beta}{2}}_{L^2(D(\sigma_1R))}\|\na(f\psi_R)\|^{\f{3}{2}\beta}_{L^2(D(\sigma_1R))}\\
&\ls R^{\al(2-\f{\beta}{2})}\|\na(f\psi_R)\|^{2-\f{\beta}{2}}_{L^2(D(\sigma_1R))}
\|\na(f\psi_R)\|^{\f{3}{2}\beta}_{L^2(D(\sigma_1R))}
\leq R^{\al(2-\f{\beta}{2})}\|\na(f\psi_R)\|^{2+\beta}_{L^2(D(\sigma_1R))},
\eali
\ee
where at the last line we have used \eqref{epoincare}. Using \eqref{e5.9} and \eqref{ebeta}, we deduce

\bes
\bali
&\int_{D(\sigma_2R)}f^{2+\beta}
\ls R^{\al(2-\f{\beta}{2})}\Big[\f{(1+\|v^r\|_{L^\i}R^\al)^2}{(\sigma_1-\sigma_2)^2R^2}\|f\|^2_{L^2(D(\sigma_1R,\sigma_2R))}\Big]^{1+\f{\beta}{2}}.
\eali
\ees

Remember $f=|\G|^q$ and set $\kappa=1+\f{\beta}{2}$, then we obtain
\bes
\bali
&\int_{D(\sigma_2R)}|\G|^{2q\kappa}
\lesssim R^{\al(2-\f{\beta}{2})}\Big[\f{1+\|v^r\|_{L^\i}R^\al}{(\sigma_1-\sigma_2)R}\Big]^{2\kappa}\Big(\int_{D(\sigma_1R,\sigma_2R)}|\G|^{2q}\Big)^{\kappa}.
\eali
\ees
For integer $j\geq 0$ and a constant $\sigma=\frac{1}{2}$, set $\sigma_2=\frac{1}{2}(1+\sigma^{j+1})$ and $\sigma_1=\frac{1}{2}(1+\sigma^j)$. Let $q=\kappa^j$, then we arrive at
\bes
\bali
&\left(\int_{D(\frac{R}{2}(1+\sigma^{j+1}))} \G^{{2}\kappa^{j+1}}dy \right)^{\frac{1}{2}(\f{1}{\kappa})^{j+1}}
\lesssim R^{\al(1-\f{\beta}{4})\lt(\f{1}{\kappa}\rt)^{j+1}} \Big[\f{1+\|v^r\|_{L^\i}R^\al}{\sigma^{j+2} R}\Big]^{(\f{1}{\kappa})^j}\left(\int_{D(\frac{R}{2}(1+\sigma^j),\frac{R}{2}(1+\sigma^{j+1}))} \G^{{2}\kappa^j}dy \right)^{\frac{1}{2}(\f{1}{\kappa})^j}.
\eali
\ees
By iterating $j$, the above inequality gives
\bes
\bali
&\left(\int_{D(\frac{R}{2}(1+\sigma^{j+1}))} \G^{{2}\kappa^{j+1}}dy \right)^{\frac{1}{2}(\frac{1}{\kappa})^{j+1}}
\lesssim R^{\al(1-\f{\beta}{4})\sum\limits^j_{i=0}\lt(\f{1}{\kappa}\rt)^{i+1}}\frac{(1+\|v^r\|_{L^\i}R^\al)^{\sum^j_{i=0}(\frac{1}{\kappa})^{i}}}
{\sigma^{{\sum^j_{i=0}(i+2)(\frac{1}{\kappa})^{i}}}R^{{\sum^j_{i=0}(\frac{1}{\kappa})^{i}}}}\left(\int_{D(R,\f{3}{4}R)}\G^{2}dy\right)^{\frac{1}{2}}.
\eali
\ees
Letting $j\rightarrow \infty$ yields that
\be\label{e2.14}
\bali
\sup\limits_{x\in D(\frac{R}{2})}|\G|&\lesssim& R^{\al(\f{2}{\beta}-\f{1}{2})}\frac{(1+\|v^r\|_{L^\i}R^\al)^{\f{2+\beta}{\beta}}}{R^{\f{2+\beta}{\beta}}}\left(\int_{D(R,R/2)}\G^{{2}}dy \right)^{\frac{1}{2}}.
\eali
\ee
Next we use the 1-dimensional Poincar\'e inequality to see that
\be\label{egamma}
\bali
&\int_{D(R,R/2)}\G^{{2}}dy=\int^{2\pi}_0\int^R_{R/2} \int^{r^\al}_{-r^\al}r^2 (v^\th)^2 dx_3 rdrd\th
\leq R^2 \int^{2\pi}_0\int^R_{R/2} \int^{r^\al}_{-r^\al}(v^\th)^2dx_3 rdrd\th \\
& \leq CR^{2+2\al}\int^{2\pi}_0\int^R_{R/2}  \int^{r^\al}_{-r^\al} (\p_{x_3} v^\th)^2dx_3 rdrd\th
\leq CR^{2+2\al}\|\na v\|^2_{L^2(\O\cap \{R/2\leq |x'|\leq R\})}.
\eali
\ee
Inserting \eqref{egamma} into \eqref{e2.14}, we can get
\bes
\sup\limits_{x\in D(\frac{R}{2})}|\G|\lesssim R^{(\al-1)\f{2}{\beta}+\f{\al}{2}}(1+\|v^r\|_{L^\i}R^\al)^{\f{2+\beta}{\beta}}\|\na v\|_{L^2(\O\cap \{R/2\leq |x'|\leq R\})}.
\ees

Using the decay estimates for $v^r$, we have
\be\label{egamma1}
\sup\limits_{x\in D(\frac{R}{2})}|\G|\lesssim R^{(\al-1)\f{2}{\beta}+\f{\al}{2}}(1+R^{\al-1/2}\ln^{1/2} R)^{\f{2+\beta}{\beta}}\|\na v\|_{L^2(\O\cap \{R/2\leq |x'|\leq R\})}.
\ee

When $1/2\leq \al< 3/4$, from \eqref{egamma1}, we find
\bes
\bali
\sup\limits_{x\in D(\frac{R}{2})}|\G|&\leq CR^{(\al-1)\f{2}{\beta}+\f{\al}{2}}(R^{\al-1/2}\ln^{1/2}R )^{\f{2+\beta}{\beta}}\\
 &\leq CR^{(2\al-\f{3}{2})\f{2}{\beta}+\f{3}{2}\al-\f{1}{2}}(\ln R)^{\f{2+\beta}{2\beta}}\rightarrow 0,\q \text{as}\ R\rightarrow\i
 \eali
\ees
by choosing sufficiently small $\beta$ such that $(2\al-\f{3}{2})\f{2}{\beta}+\f{3}{2}\al-\f{1}{2}<0,$ since $2\al-\f{3}{2}<0$. This indicates that $\G\equiv 0$ in $\O$.

So at last we have proved that $\G\equiv 0$ when $0\leq\al<3/4$, which shows that $v^\th\equiv 0$. \qed

\section*{Acknowledgments} We wish to thank the following people and organizations for their supports.  Prof. Yanyan Li for the invitation to write the review paper and for his encouragement; Prof. Hongjie Dong and Prof. Zhen Lei for the discussion and
collaboration over the years; Professors B. Carrillo, Zijin Li and Dr. Na Zhao for the recent collaboration ; Prof. Xin Yang and Mr. Chulan Zeng for going over the paper and making suggestions; Natural Science Foundation of Jiangsu Province for grant No. BK20180414; National Natural Science Foundation of China for grant No. 11801268; Simons Foundation for grant 710364.

\bibliographystyle{plain}

%\bibliography{/Users/dykim/Dropbox/01Research/oralpaper}

%%%%%%%%%%%%%%%%%%%%%%%%%%%%%%%%%%%%
%\begin{comment}

\def\cprime{$'$}

\end{document}